\documentclass[a4paper]{amsart}
\pdfoutput=1
\usepackage[utf8]{inputenc}
\usepackage[T1]{fontenc}
\usepackage{lmodern}
\usepackage{amssymb,amsxtra}
\usepackage{graphicx}
\usepackage{mathabx}
\usepackage{fancybox}
\usepackage{nicefrac,mathtools}
\usepackage[lite]{amsrefs}
\usepackage[toc,page]{appendix}
\usepackage[all,cmtip]{xy}
\renewcommand*{\MR}[1]{ \href{http://www.ams.org/mathscinet-getitem?mr=#1}{MR \textbf{#1}}}
\newcommand*{\arxiv}[1]{\href{http://www.arxiv.org/abs/#1}{arXiv: #1}}

\renewcommand{\PrintDOI}[1]{\href{http://dx.doi.org/\detokenize{#1}}{doi: \detokenize{#1}}}

\BibSpec{book}{%
  +{}  {\PrintPrimary}                {transition}
  +{,} { \textit}                     {title}
  +{.} { }                            {part}
  +{:} { \textit}                     {subtitle}
  +{,} { \PrintEdition}               {edition}
  +{}  { \PrintEditorsB}              {editor}
  +{,} { \PrintTranslatorsC}          {translator}
  +{,} { \PrintContributions}         {contribution}
  +{,} { }                            {series}
  +{,} { \voltext}                    {volume}
  +{,} { }                            {publisher}
  +{,} { }                            {organization}
  +{,} { }                            {address}
  +{,} { \PrintDateB}                 {date}
  +{,} { }                            {status}
  +{}  { \parenthesize}               {language}
  +{}  { \PrintTranslation}           {translation}
  +{;} { \PrintReprint}               {reprint}
  +{.} { }                            {note}
  +{.} {}                             {transition}
  +{} { \PrintDOI}                   {doi}
  +{} { available at \url}            {eprint}
  +{}  {\SentenceSpace \PrintReviews} {review}
}

\usepackage{enumitem} %<--advanced enumerate, which handles label and ref
\setlist[enumerate,1]{label=\textup{(\arabic*)}}% ensure enumerates in theorems are upright

\usepackage{tikz}
\usetikzlibrary{matrix,calc,arrows,decorations.pathmorphing}
\tikzset{node distance=2cm, auto}

\tikzset{cd/.style=matrix of math nodes,row sep=2em,column sep=2em, text height=1.5ex, text depth=0.5ex}
\tikzset{cdar/.style=->,auto}
\tikzset{mid/.style={anchor=mid}} % put labels on the arrow
\tikzset{narrowfill/.style={inner sep=1pt, fill=white}}% style for nodes with filled background
\tikzset{rndblock/.style={rounded corners,rectangle,draw,outer sep=0pt}}

\usepackage{tikz-cd}

\usepackage{microtype}
\usepackage[pdftitle={Aperiodicity: almost extension property and uniqueness of pseudo-expectations},
pdfauthor={Bartosz Kwa\'sniewski, Ralf Meyer},
pdfsubject={Mathematics}]{hyperref}

\theoremstyle{plain}
\newtheorem{theorem}{Theorem}
\newtheorem{lemma}[theorem]{Lemma}
\newtheorem{proposition}[theorem]{Proposition}
\newtheorem{corollary}[theorem]{Corollary}
\theoremstyle{definition}
\newtheorem{definition}[theorem]{Definition}
\theoremstyle{remark}
\newtheorem{remark}[theorem]{Remark}
\newtheorem{example}[theorem]{Example}

\numberwithin{theorem}{section}
\numberwithin{equation}{section}

\DeclareMathOperator{\Aut}{Aut}% automorphism group
\DeclareMathOperator{\Bis}{Bis}% bisections of an étale groupoid

\newcommand*{\nb}{\nobreakdash}
\newcommand*{\Star}{\(^*\)\nobreakdash-}

\newcommand*{\C}{\mathbb C}
\newcommand*{\Z}{\mathbb Z}

\newcommand*{\R}{\mathbb R}
\newcommand*{\N}{\mathbb N}
\newcommand*{\T}{\mathbb T}% circle group

\newcommand*{\Ideals}{\mathbb I}% ideal lattice
\newcommand*{\Null}{\mathcal N}% null space for positive map
\newcommand*{\Bound}{\mathbb B}%adjointable operators on a Hilbert module
\newcommand*{\Comp}{\mathbb K}%compact operators on a Hilbert module
\newcommand*{\Mat}{\mathbb M}%matrices
\newcommand{\Her}{\mathbb H}%hereditary subalgebras
\newcommand*{\red}{\mathrm r}% reduced
\newcommand*{\ess}{\mathrm{ess}}% essential
\newcommand*{\alg}{\mathrm{alg}} % algebraic
\newcommand*{\Cst}{\textup C^*}% C*-algebra
\newcommand*{\Mult}{\mathcal M}% multiplier algebra
\newcommand*{\Locmult}{\mathcal{M}_\mathrm{loc}}% local multiplier algebra
\newcommand*{\Cont}{\textup C}% continuous functions
\newcommand{\idealin}{\mathrel{\triangleleft}} % relation of being an ideal
\newcommand*{\Id}{\textup{Id}}%identity
\newcommand*{\Ad}{\textup{Ad}}%conjugation by a unitary

\newcommand*{\Hils}[1][H]{\mathcal #1}%Hilbert space
\newcommand*{\Hilm}[1][E]{\mathcal #1}%Hilbert module

\newcommand*{\A}{\mathcal A}% Fell bundles
\newcommand*{\defeq}{\mathrel{\vcentcolon=}}% used for definitions
\newcommand*{\congto}{\xrightarrow\sim}

\DeclarePairedDelimiter{\abs}{\lvert}{\rvert}% absolute value
\DeclarePairedDelimiter{\norm}{\lVert}{\rVert}% norm
\DeclarePairedDelimiter{\ket}{\lvert}{\rangle}% ket-bra notation
\DeclarePairedDelimiter{\bra}{\langle}{\rvert}% ket-bra notation
\DeclarePairedDelimiterX{\braket}[2]{\langle}{\rangle}{#1\,\delimsize\vert\,\mathopen{}#2}% inner product
\DeclarePairedDelimiterX{\BRAKET}[2]{\langle}{\rangle}{\!\delimsize\langle#1\,\delimsize\vert\,\mathopen{}#2\delimsize\rangle\!}% inner product
\DeclarePairedDelimiterX{\setgiven}[2]{\{}{\}}{#1\,{:}\,\mathopen{}#2}% set given by

\newcommand*{\dual}[1]{\widehat{#1}}% duals

\newcommand*{\s}{s}% source map of groupoids
\newcommand*{\rg}{r}% range map of groupoids

\newcommand*{\onto}{\twoheadrightarrow}

\begin{document}
\title[Aperiodicity: the almost extension property and pseudo-expectations]{Aperiodicity: the almost extension property and uniqueness of pseudo-expectations}

\author{Bartosz Kosma Kwa\'sniewski}
\email{bartoszk@math.uwb.edu.pl}
 \address{Faculty of Mathematics\\
   University  of Bia\l ystok\\
   ul.\@ K.~Cio\l kowskiego 1M\\
   15-245 Bia\l ystok\\
   Poland}

\author{Ralf Meyer}
\email{rmeyer2@uni-goettingen.de}
\address{Mathematisches Institut\\
 Georg-August-Universit\"at G\"ottingen\\
 Bunsenstra\ss e 3--5\\
 37073 G\"ottingen\\
 Germany}

\begin{abstract}
  We prove implications among the conditions in the title for
  general \(\Cst\)\nb-inclusions \(A\subseteq B\), and we also
  relate this to several other properties in case~\(B\) is a crossed
  product for an action of a group, inverse semigroup or an \'etale
  groupoid on~\(A\).  We show that an aperiodic
  \(\Cst\)\nb-inclusion has a unique pseudo-expectation.  If, in
  addition, the unique pseudo-expectation is faithful, then~\(A\)
  supports~\(B\) in the sense of the Cuntz preorder.  The almost
  extension property implies aperiodicity, and the converse holds
  if~\(B\) is separable.  A crossed product inclusion
  has the almost extension property if and only if
  the dual groupoid of the action is topologically principal.
  Topologically free actions are always aperiodic.  If~\(A\) is
  separable or of Type~I, then topological freeness, aperiodicity
  and having a unique pseudo-expectation are equivalent to the
  condition that~\(A\) detects ideals in all
  \(\Cst\)\nb-algebras~\(C\) with \(A\subseteq C \subseteq B\).  If,
  in addition, \(B\) is separable, then all these conditions are
  equivalent to the almost extension property.
\end{abstract}


\thanks{BKK was partially supported by the National Science Center
  (NCN), Poland, grant no.~2019/35/B/ST1/02684}

% 20M18, 22A22
\maketitle
% \setcounter{tocdepth}{1}
% \tableofcontents

\section{Introduction}
\label{sec:introduction}

Many important \(\Cst\)\nb-algebras may be described using crossed
products for group actions and their generalisations.  This makes it
important to describe the ideal structure of crossed products or to
decide whether they are purely infinite.  Very satisfactory criteria
for this were developed around 1980 by Olesen and Pedersen
\cites{Olesen-Pedersen:Applications_Connes,
  Olesen-Pedersen:Applications_Connes_2,
  Olesen-Pedersen:Applications_Connes_3}, Kishimoto
\cites{Kishimoto:Outer_crossed, Kishimoto:Freely_acting}, and
Rieffel~\cite{Rieffel:Actions_finite}.  These were recently extended
in~\cite{Kwasniewski-Meyer:Aperiodicity} from ordinary group actions
to Fell bundles over groups.  Partial, twisted actions of groups are
a special case of this.  Another important generalisation of this
theory is to actions of inverse semigroups by Hilbert bimodules.
Such actions and their crossed products model \(\Cst\)\nb-algebras
associated to Fell bundles over \'etale groupoids.  The main new
difficulties in this more general setting come from the
non-Hausdorffness of locally compact groupoids.  These were overcome
recently in~\cite{Kwasniewski-Meyer:Essential} where we proposed a
construction of an \emph{essential crossed product}.  The latter
coincides with the reduced crossed products for actions of groups or
Hausdorff \'etale groupoids.

The articles by Olesen--Pedersen already study a rather large number
of closely related properties for group actions.  Two properties
have, however, only come into focus more recently.  The first is the
\emph{uniqueness of pseudo-expectations}, which is used, for
instance, in \cites{Pitts:Regular_I, Pitts:Regular_II,
  Pitts-Zarikian:Unique_pseudoexpectation,
  Zarikian:Unique_expectations}.  The second is the \emph{almost
  extension property} for a \(\Cst\)\nb-inclusion \(A\subseteq B\),
which says that there is a dense set of pure states on~\(A\) that
extend uniquely to~\(B\).  This property is used, for instance, in
\cites{Nagy-Reznikoff:Pseudo-diagonals, Exel-Pitts:Weak_Cartan}.  In
this article, we relate these two properties to aperiodicity,
topological freeness and to the property that a
\(\Cst\)\nb-subalgebra \(A\subseteq B\) \emph{detects ideals in all
  intermediate \(\Cst\)\nb-algebras} \(A\subseteq C\subseteq B\).
Our findings are summarised in the diagram in
Figure~\ref{fig:diagram}.
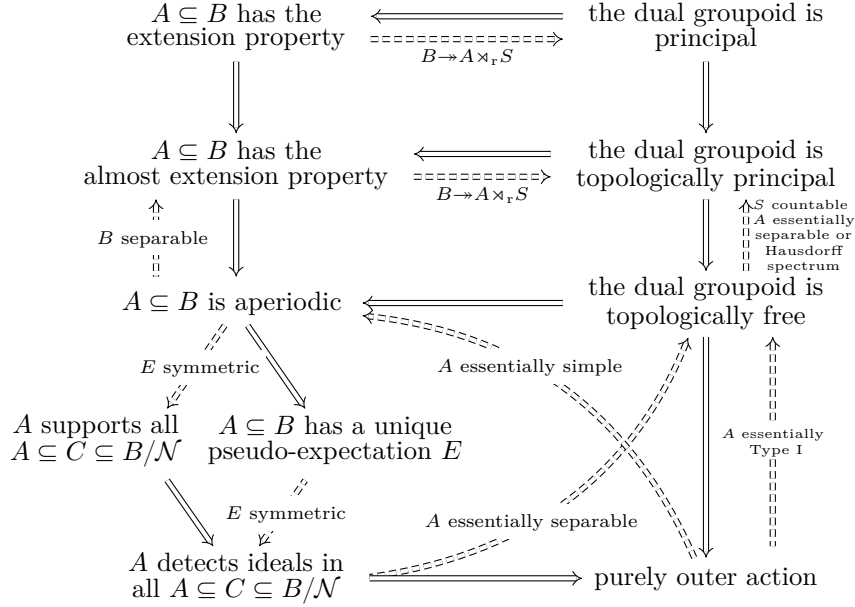
\begin{figure}[htbp]
  \centering
  \begin{tikzcd}[column sep=normal, row sep=huge]
    {\renewcommand{\arraystretch}{0.7}
      \begin{array}{c} A\subseteq B\text{ has the}\\
        \text{extension property}
      \end{array}}
    \ar[d, Rightarrow,  start anchor={[xshift=0em,  yshift=-1ex]},
    end anchor={[xshift=0em, yshift=1ex ]}]
    \ar[r, shift left=1ex, Leftarrow]
    \ar[r, dashed, shift right=1ex, Rightarrow,
    "B\onto A\rtimes_\red S"']
    &
    {\renewcommand{\arraystretch}{0.7}
      \begin{array}{c} \text{the dual groupoid is}\\
        \text{principal}
      \end{array}}
    \ar[d, Rightarrow, start anchor={[xshift=0em,  yshift=-1ex]},
    end anchor={[xshift=0em, yshift=1ex ]}]
    \\
    {\renewcommand{\arraystretch}{0.7}
      \begin{array}{c} A\subseteq B\text{ has the}\\
        \text{almost extension property}
      \end{array}}
    \ar[d, Rightarrow,  start anchor={[xshift=0em,  yshift=-1ex]},
    end anchor={[xshift=0em, yshift=0.5ex ]}]
    \arrow[d, dashed, Leftarrow, "B\text{ separable }" description,
    start anchor={[xshift=-3em,  yshift=-1ex]}, end anchor={[xshift=-3em, yshift=0.5ex ]}]
    \ar[r, shift left=1ex, Leftarrow]
    \ar[r, dashed, shift right=1ex, Rightarrow,
    "B\onto A\rtimes_\red S"']
    &
    {\renewcommand{\arraystretch}{0.7}
      \begin{array}{c} \text{the dual groupoid is}\\
        \txt{topologically principal}
      \end{array}}
    \ar[d, Rightarrow, start anchor={[xshift=0,  yshift=-1ex]},
    end anchor={[xshift=0, yshift=1ex ]}]
	%	/tikz/
    \arrow[d, dashed, Leftarrow,
		"\scriptsize\substack{S\text{ countable } \\ A 	\text{ essentially}\\	\text{separable or}\\  \text{Hausdorff} \\ \text{spectrum}}"
		,
    start anchor={[xshift=1.5em,  yshift=-1ex]}, end anchor={[xshift=1.5em, yshift=1ex ]}]
		%/tikz/inner sep=0.5ex
    \\
    A\subseteq B \text{ is  aperiodic }
    \ar[d,  Rightarrow, start anchor={[xshift=1ex, yshift=0ex]}, end anchor={[xshift=6ex, yshift=1ex ]}]
    \ar[d, dashed, Leftarrow, "\,\,E\text{ symmetric}" description, Rightarrow, start anchor={[xshift=-1ex, yshift=0ex]}, end anchor={[xshift=-6ex, yshift=1ex ]}]
    \ar[r, Leftarrow]
    &
    {\renewcommand{\arraystretch}{0.7}
      \begin{array}{c} \text{the dual groupoid is}
        \\
        \txt{topologically free}
      \end{array}}
    \arrow[dd, Rightarrow, start anchor={[xshift=0ex, yshift=-1ex]}, end anchor={[xshift=0ex, yshift=0ex ]}]
    \arrow[ddl, dashed, Leftarrow, pos=0.6, bend left,
    "A\text{ essentially separable }" description, shorten <= 0.5em]
		\arrow[dd, dashed, Leftarrow,  "\scriptsize \substack{A\text{ essentially} \\ \text{ Type I}}" description,
    start anchor={[xshift=2.5em,  yshift=-1ex]}, end anchor={[xshift=2.5em, yshift=1ex ]}]
    \\
    \renewcommand{\arraystretch}{0.7}
    \begin{tabular}{c}
      $A$\text{ supports all}\\
      $A\subseteq C\subseteq B/\Null$
    \end{tabular}
    \begin{tabular}{c}
      \text{$A\subseteq B$ has a unique}\\
      \text{pseudo-expectation \(E\)}
    \end{tabular}
    \\
    {\renewcommand{\arraystretch}{0.7}
      \begin{array}{c}
        A \text{ detects ideals in}\\
        \text{all }A\subseteq C\subseteq B/\Null
      \end{array}}
    \arrow[u, dashed, Leftarrow, "\,\,E\text{ symmetric}" description, start anchor={[xshift=2ex,  yshift=1ex]}, end anchor={[xshift=6ex, yshift=-1.5ex ]}]
    \ar[u,  Leftarrow, start anchor={[xshift=-2ex,  yshift=1ex]}, end anchor={[xshift=-6ex, yshift=-1.5ex ]}]
    \arrow[r, Rightarrow]
    &
    \text{purely outer action}
    \arrow[uul, dashed, Rightarrow, bend right, pos=0.6,
    "A\text{ essentially simple}" description, end anchor={[xshift=0ex, yshift=-1ex]}]
  \end{tikzcd}

  \caption{Implications among properties of \(\Cst\)\nb-inclusions.
    For the left column, let \(A\subseteq B\) be a
    \(\Cst\)\nb-inclusion, let \(E\colon B\to I(A)\) be a
    pseudo-expectation and let~\(\Null\) be the largest two-sided
    ideal in~\(B\) with \(\Null\subseteq \ker E\).  For the whole
    diagram, assume that~\(B\) is an exotic crossed product for an
    action of an inverse semigroup \(S\) and let \(E\) be the
    canonical pseudo-expectation.}
  \label{fig:diagram}
\end{figure}

  If~\(B\)
is the reduced crossed product for a group action or an essential
crossed product for an action of an inverse semigroups or an \'etale
groupoid, and~\(B\) is separable,   then all the properties
above are equivalent by Theorem~\ref{the:cycle_of_equivalences}.
Here the action may be by automorphisms or by Hilbert bimodules.
The latter case corresponds to Fell bundles.

The almost extension property is introduced by Nagy and
Reznikoff~\cite{Nagy-Reznikoff:Pseudo-diagonals}.  It is weaker than
the well known and extensively studied \emph{extension property} due
to Anderson~\cite{Anderson:Extensions_states}.  Using a criterion by
Anderson for pure states to extend uniquely, we prove the following:
if a \(\Cst\)-inclusion \(A\subseteq B\) has the almost extension
property, then it is \emph{aperiodic} (we generalised aperiodicity
from actions to general \(\Cst\)\nb-inclusions
in~\cite{Kwasniewski-Meyer:Aperiodicity}), and the converse holds
when~\(B\) is separable (see
Theorem~\ref{the:almost_extension_aperiodic}).  When \(B\) is a
crossed product (full or reduced) of an action of a group, inverse
semigroup or an étale groupoid, then we characterise pure states
that extend uniquely in terms of the isotropy of a \emph{dual
  groupoid}.  As a consequence, we show that \(A\subseteq B\) has
the (almost) extension property if and only if the dual groupoid is
(topologically) principal (see
Theorem~\ref{the:almost_extension_top_principal}).  This is a far
reaching generalisation of the recent result of
Zarikian~\cite{Zarikian:Pure_extension}*{Theorem 2.4} proved for
group actions by automorphisms on unital \(\Cst\)-algebras.

We call an action of a group, inverse semigroup or étale groupoid
\emph{topologically free} if the dual groupoid of the action is
topologically free in the sense of
\cite{Kwasniewski-Meyer:Essential}.  In general, this is a weaker
condition than being effective or topologically principal.  For
group actions, it is (at least formally) weaker than the condition
used by
Archbold--Spielberg~\cite{Archbold-Spielberg:Topologically_free} (we
refer to \cite{Kwasniewski-Meyer:Essential}*{Section~2.4} for a
careful comparison of these properties).  We prove that
topologically free actions are always aperiodic (see
Corollary~\ref{cor:top_free_vs_aperiodic}).  So far, this was only
known for actions on separable \(\Cst\)\nb-algebras, and the proof
for actions by Hilbert bimodules (Fell bundles) was rather indirect.
Our proof is based on the technique of excising states
from~\cite{Akemann-Anderson-Pedersen:Excising}.

Let \(I(A)\) be Hamana's injective
envelope~\cite{Hamana:Injective-Envelope-Cstar}.  A
\emph{pseudo-expectation} for a \(\Cst\)\nb-inclusion
\(A\subseteq B\) is defined as a completely positive contraction
\(E\colon B \to I(A)\) that extends the identity map on~\(A\)
(see~\cite{Pitts:Regular_I}).  Unlike genuine conditional
expectations, pseudo-expectations always exist because~\(I(A)\) is
injective.  Pitts and Zarikian studied extensively the case when
there is a unique pseudo-expectation
\cites{Pitts-Zarikian:Unique_pseudoexpectation,
  Zarikian:Unique_expectations}.  We prove that any aperiodic
\(\Cst\)\nb-inclusion \(A\subseteq B\) has a unique
pseudo-expectation \(E\colon B\to I(A)\).  We also improve the
results in~\cite{Kwasniewski-Meyer:Essential} about general
aperiodic inclusions.  Namely, for any aperiodic inclusion
\(A\subseteq B\), we prove that for any \(b\in B\) with \(b\ge 0\)
and \(E(b)\neq 0\), there is \(a\in A\) with \(a \precsim b\) in the
Cuntz preorder on~\(B\).  In particular, if the unique
pseudo-expectation is faithful, this says that~\(A\)
\emph{supports}~\(B\).  This condition plays a crucial role in the
study of pure infiniteness (see \cites{Kwasniewski:Crossed_products,
  Kwasniewski-Meyer:Essential, Kwasniewski-Meyer:Aperiodicity}).
In~\cite{Kwasniewski-Meyer:Essential}, this is proven when there is
a ``supportive'' generalised expectation.  The new information is
that all pseudo-expectations are ``supportive''.  The results in
this paragraph are contained in
Theorem~\ref{thm:aperiodicity_implies_unique_expectation}, which
collects various consequences of aperiodicity.

We say that~\(A\) \emph{detects ideals} in~\(B\) if \(J\cap A = 0\)
for an ideal \(J\subseteq B\) implies \(J=0\).  If an action of a
discrete group~\(G\) on a \(\Cst\)\nb-algebra~\(A\) is topologically
free, then~\(A\) detects ideals in \(A\rtimes_\red G\)
(see~\cite{Archbold-Spielberg:Topologically_free}).  The converse to
this usually fails, except for very special groups like~\(\Z\).  We
show, however, that an action has to be topologically free if~\(A\)
is separable or of Type~I and~\(A\) detects ideals in~\(C\) for all
intermediate \(\Cst\)\nb-algebras \(A\subseteq C\subseteq B\).  This
remains true for actions of inverse semigroups or étale groupoids.
Moreover, it suffices here to consider intermediate
\(\Cst\)\nb-subalgebras that are crossed products associated to a
subgroup, inverse subsemigroup, or a subgroupoid, depending on the
kind of action in question (see Propositions
\ref{pro:intermediate_detection_to_top-free_isg}
and~\ref{pro:intermediate_detection_to_top-free_groupoid}).  If, in
addition, the inverse semigroup that acts is countable, then the
dual groupoid is topologically free if and only if it is
topologically principal by
\cite{Kwasniewski-Meyer:Essential}*{Theorem~6.13}.  And we show that
this is equivalent to the almost extension property.  Thus all
properties studied here are equivalent for actions of countable
inverse semigroups on \(\Cst\)\nb-algebras that are separable or of
Type~I.  These equivalences are collected in
Theorem~\ref{the:cycle_of_equivalences}.

Many of our results are already interesting for group actions by
automorphisms.  In this case, the essential crossed product is the
same as the reduced crossed product, and many technical issues do
not occur.  We expect that some readers are not familiar with the
more general actions of inverse semigroups and étale groupoids.  We
have written the article in such a way that these readers should be
able to get along by just ignoring these generalisations.  To make
the article easier to digest for such readers, we introduce the key
concepts such as aperiodicity and topological freeness first for
group actions and then sketch only briefly how they must be adapted
to treat actions of inverse semigroups and étale groupoids.  Readers
who need the more general theory should
consult~\cite{Kwasniewski-Meyer:Essential} and the references there
for a more thorough introduction of inverse semigroup actions by
Hilbert bimodules, Fell bundles over étale groupoids and their
crossed products.

The paper is organised as follows.  We start with an introduction of
aperiodic actions and aperiodic inclusions in
Section~\ref{sec:aperiodic_actions}.  In Section
\ref{sec:aperiodic_pseudo-expectation} we prove that any aperiodic
inclusion \(A\subseteq B\) has a unique pseudo-expectation and that
all pseudo-expectations have the technical property of being
\emph{supportive} introduced in~\cite{Kwasniewski-Meyer:Essential}.
This strengthens our previous results and gives natural simplicity
and pure infiniteness criteria for general aperiodic
\(\Cst\)-inclusions.  In Section~\ref{sec:top_free_aperiodic}, we
elaborate on topological freeness for actions of groups, inverse
semigroups and \'etale groupoids.  We show that topologically free
actions are aperiodic.  The key result concerns a single Hilbert
bimodule (see Theorem~\ref{the:top_free_vs_aperiodic}).  The proof
uses the concept of a net of elements excising a state
from~\cite{Akemann-Anderson-Pedersen:Excising}.  This concept is
also used in Section~\ref{sec:aperiodicity_almost extension}, where
we discuss the almost extension property and, along the way, the
extension property, for pure states.  We prove that the almost
extension property implies aperiodicity for arbitrary
\(\Cst\)\nb-inclusions, and the converse holds in the separable
case.  We show that a crossed product inclusion has the almost
extension property if and only if the dual groupoid is topologically
principal.  In Section~\ref{sec:detection_in_intermediate} we prove
that an action of an inverse semigroup or an \'etale groupoid is
topologically free if~\(A\) detects ideals in certain intermediate
\(\Cst\)\nb-algebras \(A\subseteq C\subseteq B\), where~\(B\) is an
essential crossed product and~\(A\) contains an essential ideal that
is separable or of Type~I.  In Section~\ref{sec:various_conditions},
we summarise the results of this paper presented in Figure~\ref{fig:diagram}.
Theorems \ref{the:cycle_of_equivalences}
and~\ref{the:cycle_of_equivalences_2} show that various properties
are equivalent for actions of inverse semigroups or étale, locally
compact groupoids on \(\Cst\)\nb-algebras that are ``essentially''
separable, simple, or of Type~I.  These equivalences allow to weaken
the assumptions and strengthen the conclusions in a number of
results in the literature.  We use them in Section~\ref{sec:Cartan}
to charactaterise Cartan inclusions and \(\Cst\)\nb-diagonals.

\section{Aperiodic actions and aperiodic inclusions}
\label{sec:aperiodic_actions}

Aperiodicity is one of the key concepts in this article.  In this
section, we discuss the definition of aperiodic inclusions and how
it translates to aperiodicity for actions of groups, inverse
semigroups and \'etale groupoids.  Here we understand actions in a
broad sense as Fell bundles over such objects. This section contains no new results.

Let~\(A\) be a \(\Cst\)\nb-algebra.  Let
\begin{align*}
  \Her(A)
  &\defeq
    \{ \text{non-zero, hereditary $\Cst$-subalgebras of~} A\},\\
  A^+_1
  &\defeq \setgiven{a\in A}{a\ge0,\ \norm{a}=1}.
\end{align*}

\begin{definition}[\cite{Kwasniewski-Meyer:Aperiodicity}*{Definition 4.1}]
  \label{def:aperiodic_module}
  Let~\(X\) be a normed \(A\)\nb-bimodule.  We say that \(x\in X\)
  satisfies \emph{Kishimoto's condition} if, for any
  \(D\in \Her(A)\) and any \(\varepsilon>0\), there is
  \(a\in D^+_1\) with \(\norm{a x a}<\varepsilon\).  We call~\(X\)
  \emph{aperiodic} if Kishimoto's condition holds for all \(x\in X\)
  (we renamed this last condition in
  \cite{Kwasniewski-Meyer:Essential}*{Definition 5.9}).
\end{definition}

\begin{example}
  \label{exa:Hilbi_from_auto}
  \label{exa:Kishimoto_automorphism}
  An automorphism \(\alpha\in \Aut(A)\) defines a Hilbert
  \(A\)\nb-bimodule~\(A_\alpha\) as follows: it is~\(A\) as a vector
  space, with the bimodule structure
  \[
    a\cdot x\cdot b \defeq a x \alpha(b)
  \]
  for \(a,b\in A\) \(x\in A_\alpha\), and with the left and right
  inner products
  \[
    \BRAKET{x}{y}\defeq xy^*,\qquad
    \braket{x}{y}\defeq \alpha^{-1}(x^*y)
  \]
  for \(x,y\in A_\alpha\).  As a Hilbert bimodule, \(A_\alpha\) is
  given the norm
  \(\norm{x} \defeq \norm{\braket{x}{x}}^{1/2} =
  \norm{\BRAKET{x}{x}}^{1/2}\).  This is equal to the
  \(\Cst\)\nb-norm on~\(A\).

  Kishimoto's condition for the Hilbert
  \(A\)\nb-bimodule~\(A_\alpha\) is exactly the condition introduced
  by Kishimoto in \cite{Kishimoto:Outer_crossed}*{Lemma~1.1}.  By
  \cite{Kishimoto:Freely_acting}*{Theorem~2.1}, \(A_\alpha\) is
  aperiodic if and only if there is no \(\alpha\)\nb-invariant ideal
  \(J\subseteq A\) for which the Borchers spectrum of~\(\alpha|_J\)
  is equal to \(\{1\} \subseteq \T\).  An automorphism with this
  property is often called freely acting or \emph{properly outer}
  (see, for instance, \cites{Kennedy-Schafhauser:noncomm_boundaries,
    Zarikian:Unique_expectations}).  A related condition due to
  Elliott asks for \(\norm{\alpha|_J - \Ad(u)} = 2\) for all
  \(\alpha\)\nb-invariant ideals \(J\subseteq A\) and all unitary
  multipliers~\(u\) of~\(J\).  Kishimoto's condition implies
  Elliott's, and the converse also holds if~\(A\) is separable (see
  \cite{Olesen-Pedersen:Applications_Connes_3}*{Theorem~6.6}).  An
  even weaker condition is \emph{pure outerness}, which says only that
  \(\alpha|_J \neq \Ad(u)\) for all \(\alpha\)\nb-invariant ideals
  \(J\subseteq A\) and all unitary multipliers~\(u\) of~\(J\)
  (see~\cite{Rieffel:Actions_finite}).  If~\(A\) is simple, an
  automorphism is purely outer if and only if it is outer, and then
  it is already properly outer by Kishimoto's Theorem
  from~\cite{Kishimoto:Outer_crossed}.  Purely outer automorphisms
  of Type~I \(\Cst\)\nb-algebras are properly outer as well.  A
  purely outer automorphism of a separable \(\Cst\)\nb-algebra that
  is not properly outer is described in
  \cite{Kwasniewski-Meyer:Aperiodicity}*{Example~2.14}.
\end{example}

\begin{definition}[\cite{Kwasniewski-Meyer:Essential}*{Definition 5.14}]
  \label{def:aperiodic_inclusion}
  A \(\Cst\)\nb-inclusion \(A\subseteq B\) is \emph{aperiodic} if
  the Banach \(A\)\nb-bimodule \(B/A\) is aperiodic, that is, if for
  every \(x\in B\), \(D\in \Her(A)\) and \(\varepsilon>0\), there are
  \(a\in D^+_1\) and \(y\in A\) with \(\norm{a x a - y}<\varepsilon\).
\end{definition}

\begin{proposition}
  \label{pro:aperiodic_group_action}
  Let \(\alpha\colon G\to\Aut(A)\) be a group action of a discrete
  group~\(G\).  Form Hilbert bimodules~\(A_{\alpha_g}\) for
  \(g\in G\) as in
  Example~\textup{\ref{exa:Kishimoto_automorphism}}.  Let
  \(B\defeq A\rtimes G\) be the full crossed product.  The canonical
  inclusion \(A\hookrightarrow B\) is aperiodic if and only if the
  normed \(A\)\nb-bimodules~\(A_{\alpha_g}\) are aperiodic for all
  \(\alpha\in G\setminus\{1\}\).
\end{proposition}

\begin{proof}
  The main point is the following.  Let~\(X\) be a normed
  \(A\)\nb-bimodule and let~\(X_i\) for \(i\in I\) be subbimodules
  such that \(\sum X_i\) is dense in~\(X\).  Give~\(X_i\) the norm
  from~\(X\).  Then~\(X\) is aperiodic if and only if each~\(X_i\)
  is aperiodic.  This follows easily from
  \cite{Kwasniewski-Meyer:Aperiodicity}*{Lemma~4.2} (see also
  \cite{Kwasniewski-Meyer:Essential}*{Lemma~5.12}).  For \(g\in G\),
	let $u_g\in \Mult(B)$ be the corresponding unitary in the multplier algebra
	of the crossed product. Then \(A\cdot u_g \subseteq B\) is an \(A\)\nb-subbimodule that is
  isomorphic to~\(A_{\alpha_g}\) as a bimodule because
  \(a\cdot (b\cdot u_g)\cdot c = a b \alpha_g(c) \cdot u_g\) for all
  \(a,b,c\in A\).  We claim that this isomorphism remains isometric
  as a map to \(B/A\), that is, \(\norm{a\cdot u_g + b} \ge \norm{a}\)
  for all \(a,b\in A\), \(g\in G\setminus\{1\}\).  The proof of the
  claim uses the regular representation
  \(A\rtimes G \onto A\rtimes_\red G \subseteq \Bound(L^2(G,A))\)
  and that \(a\cdot  u_g\) and~\(b\) are orthogonal in \(L^2(G,A)\).  Since
  \(\sum_{g\in G\setminus\{1\}} A\cdot u_g\) is dense in~\(B/A\), the
  statement about sums of bimodules shows that the inclusion
  \(A\subseteq B\) is aperiodic if and only if each~\(A_{\alpha_g}\)
  for \(g\in G\setminus\{1\}\) is aperiodic.
\end{proof}

The same argument works when we replace \(A\rtimes G\) by
\(A\rtimes_\red G\) or any \emph{exotic crossed product}, that is, a
\(\Cst\)\nb-algebra~\(B\) with surjective \Star{}homomorphisms
\[
  A\rtimes G \to B \to A\rtimes_\red G
\]
that compose to the canonical quotient map
\(A\rtimes G \to A\rtimes_\red G\).  We could also allow twisted
actions or partial actions of~\(G\).  We turn right away to the most
general kind of group actions, namely, Fell bundles.

A \emph{Fell bundle~\(\A\) over a discrete group}~\(G\) is a family
of Banach spaces \((A_g)_{g\in G}\) with bilinear, associative
multiplication maps \(A_g \times A_h \to A_{g h}\) and
conjugate-linear, antimultiplicative involutions
\(A_g \to A_{g^{-1}}\) satisfying natural properties that turn the
unit fibre \(A\defeq A_1\) into a \(\Cst\)-algebra, each~\(A_g\)
into a Hilbert \(A\)\nb-bimodule, and the direct sum
\(\bigoplus_{g\in G} A_g\) into a \Star{}algebra.  The full and the
reduced section \(\Cst\)\nb-algebras, \(\Cst(G,\A)\) and
\(\Cst_\red(G,\A)\), are defined as \(\Cst\)-completions of
\(\bigoplus_{g\in G} A_g\).  A \emph{Fell bundle} over~\(G\) is
called \emph{aperiodic} if~\(A_g\) for \(g\in G\setminus\{1\}\) is
aperiodic
\cite{Kwasniewski-Szymanski:Pure_infinite}*{Definition~4.1}.  The
argument in the proof of Proposition
\ref{pro:aperiodic_group_action} shows that the Fell bundle is
aperiodic if and only if the inclusion of~\(A\) into \(\Cst(G,\A)\)
is aperiodic.  Here we may replace \(\Cst(G,\A)\) by any
\(\Cst\)\nb-algebra with surjective \Star{}homomorphisms
\(\Cst(G,\A) \onto B \onto \Cst_\red(G,\A)\) that compose to the
canonical \Star{}homomorphism \(\Cst(G,\A) \onto \Cst_\red(G,\A)\).

\begin{example}
  For any kind of generalised action of a group~\(G\), the ``crossed
  product''~\(B\) should be \(G\)\nb-graded, that is, it should come
  with closed linear subspaces \(B_g\subseteq B\) for \(g\in G\)
  that satisfy \(B_g \cdot B_h \subseteq B_{g h}\) and
  \(B_g^* = B_{g^{-1}}\) for \(g,h\in G\) and that \(\sum B_g\) is
  dense in~\(B\).  Then \((B_g)_{g\in G}\) with the multiplication
  and involution from~\(B\) is a Fell bundle over~\(G\).  And the
  maps \(B_g \to B\) form a Fell bundle representation.  So they
  induce a surjective \Star{}homomorphism
  \(\Cst(G,(B_g)_{g\in G}) \onto B\).  A \(G\)\nb-grading is called
  \emph{topological} if there is also a surjective
  \Star{}homomorphism \(B \onto \Cst_\red(G,\A)\) as above so that
  the composite \Star{}homomorphism
  \(\Cst(G,\A) \onto \Cst_\red(G,\A)\) is the canonical quotient map
  (see, for instance, \cite{Exel:Partial_dynamical}).  Crossed
  products for twisted (partial) actions of~\(G\) are
  \(G\)\nb-graded by construction.  Thus twisted (partial) actions
  define Fell bundles. The full and reduced crossed product
  \(\Cst\)\nb-algebras are naturally isomorphic to the full and
  reduced section \(\Cst\)\nb-algebras of the corresponding Fell
  bundle.
\end{example}

Now let~\(S\) be an inverse semigroup with unit \(1\in S\).  An
\emph{action of the inverse semigroup}~\(S\) on a
\(\Cst\)\nb-algebra~\(A\) by Hilbert bimodules consists of Hilbert
\(A\)\nb-bimodules~\(\Hilm_t\) for \(t\in S\) and unitary
multiplication maps
\(\mu_{t,u}\colon \Hilm_t \otimes_A \Hilm_u \to \Hilm_{t u}\) for
\(t,u\in S\), such that~\(\mu_{t,u}\) is associative,
\(\Hilm_1 = A\), and \(\mu_{1,t}\) and~\(\mu_{t,1}\) are the
canonical maps for all \(t\in S\) (see
\cite{Buss-Meyer:Actions_groupoids}*{Definition~4.7}).  Such an
action is equivalent to a saturated Fell bundle over~\(S\) and so it
has a full and a reduced section \(\Cst\)\nb-algebra
(see~\cites{Buss-Exel-Meyer:Reduced, Exel:noncomm.cartan,
  Kwasniewski-Meyer:Essential}).  We think of these as
generalisations of full and reduced crossed products for group
actions and denote them by \(A\rtimes S\) and \(A\rtimes_\red S\),
respectively.  In addition, we shall also use the essential crossed
product \(A\rtimes_\ess S\) defined
in~\cite{Kwasniewski-Meyer:Essential}, which in general is a
quotient of \(A\rtimes_\red S\).  We will discuss its definition
when it becomes relevant.

An important difference between crossed products for group and
inverse semigroup actions is that the images of~\(\Hilm_t\) in
\(A\rtimes S\) for \(t\in S\) are no longer linearly independent.
The intersection of~\(\Hilm_t\) with~\(A\) in \(A\rtimes S\) is
equal to the following ideal in~\(A\):
\begin{equation}
  \label{eq:Itu}
  I_{1,t} \defeq \overline{\sum_{v \le t,1} \s(\Hilm_v)}.
\end{equation}
Here \(\s(\Hilm_v)\) is the closed ideal generated by the inner
products \(\braket{x}{y}\) for \(x,y\in\Hilm_v\), and ``\(\le\)'' is
the standard partial order on \(S\) (\(v \le t,1\) means that \(v\) is an idempotent and \(v=tv\)).  Let
\[
  I_{1,t}^\bot \defeq \setgiven{x\in A}{x\cdot I_{1,t}=0}
\]
be the annihilator of the ideal~\(I_{1,t}\) in~\(A\).  If \(S=G\) is
a group, then for each \(t\in G\setminus\{1\}\) the sum in
\eqref{eq:Itu} is empty and hence \(I_{1,t}=0\) and
\(I_{1,t}^\bot=A\).  Thus \(A\rtimes_{\red} G=A\rtimes_\ess G\).  The
following proposition generalises
Proposition~\ref{pro:aperiodic_group_action} to inverse semigroup
actions:

\begin{proposition}[\cite{Kwasniewski-Meyer:Essential}*{Proposition~6.3,
    Definition 6.1}]
  \label{pro:aperiodic_isg_action}
  Let~\(\Hilm\) be an action of an inverse semigroup on a
  \(\Cst\)\nb-algebra~\(A\).  Let~\(B\) be a \(\Cst\)\nb-algebra
  with surjective \Star{}homomorphisms
  \[
    A\rtimes S \onto B\onto A\rtimes_\ess S
  \]
  that compose to the quotient map
  \(A\rtimes S \onto A\rtimes_\ess S\).  The inclusion
  \(A\subseteq B\) is aperiodic if and only if the
  \(A\)\nb-bimodules \(\Hilm_t\cdot I_{1,t}^\bot\) for \(t\in S\)
  are aperiodic.  In this case, we call the \emph{action~\(\Hilm\)
    aperiodic}.
\end{proposition}

We call a \(\Cst\)\nb-algebra~\(B\) as in
Proposition~\ref{pro:aperiodic_isg_action} an \emph{exotic crossed
  product} for the action of~\(S\) on~\(A\) (see \cite{Kwasniewski-Meyer:Essential}*{Section 4.2}).

Next we briefly discuss \(\Cst\)\nb-algebras associated to étale
groupoids.  A topological groupoid~\(G\) is étale if its range and
source maps \(r\) and~\(s\) are local homeomorphisms.  An open set
\(U\subseteq G\) is a \emph{bisection} of~\(G\) if \(r|_U\)
and~\(s|_U\) are injective.  The bisections in~\(G\) form a unital
inverse semigroup \(\Bis(G)\), where the space of units is the unit
bisection.

\begin{example}
  \label{exa:groupoid_Cstar_aperiodic}
  Let~\(G\) be an étale groupoid with locally compact Hausdorff object
  space~\(X\).  The full groupoid \(\Cst\)\nb-algebra \(\Cst(G)\)
  contains \(\Cont_0(X)\) as a \(\Cst\)\nb-subalgebra.  By
  construction, the linear span of \(\Cont_0\)\nb-functions on
  bisections in~\(G\) is dense in \(\Cst(G)\).  For each bisection
  \(U\subseteq G\), the corresponding subspace
  \(\Cont_0(U) \subseteq \Cst(G)\) is a Hilbert
  \(\Cont_0(X)\)\nb-bimodule.  The bimodule structure is
  \((f_1\cdot f_2\cdot f_3)(g) \defeq f_1(r(g)) f_2(g) f_3(s(g))\)
  for all \(f_1,f_3\in \Cont_0(X)\), \(f_2\in\Cont_0(U)\),
  \(g\in G\).  The Banach spaces \(\Hilm_U\defeq \Cont_0(U)\) form
  an action of \(\Bis(G)\) on \(\Cont_0(X)\) by Hilbert bimodules,
  whose full section \(\Cst\)\nb-algebra is naturally isomorphic to
  \(\Cst(G)\).  Given a bisection \(U\subseteq G\), the submodule
  \(\Hilm_U\cdot I_{1,U}^\bot\) in
  Proposition~\ref{pro:aperiodic_isg_action} is
  \(\Cont_0(U\setminus \overline{X})\).  Therefore,
  Proposition~\ref{pro:aperiodic_isg_action} says that the inclusion
  \(\Cont_0(X) \subseteq \Cst(G)\) is aperiodic if and only if the
  \(\Cont_0(X)\)-bimodules \(\Cont_0(U\setminus \overline{X})\) with
  the supremum norm are aperiodic for all bisections~\(U\) of~\(G\).
  It is easy to see that this happens if and only if there is no
  non-empty open subset \(U\subseteq G\setminus X\) with
  \(r|_U = s|_U\), if and only if the set of \(g\in G\) with
  \(r(g) \neq s(g)\) is dense in \(G\setminus \overline{X}\).  The
  same argument works if we replace the full groupoid
  \(\Cst\)\nb-algebra by the reduced one or by the essential one
  defined in~\cite{Kwasniewski-Meyer:Essential}.
\end{example}

We name the condition that characterises aperiodicity in the above
example:

\begin{definition}[\cite{Kwasniewski-Meyer:Essential}*{Definition~2.20}]
  \label{def:topologically_free_groupoid}
  An étale groupoid~\(G\) with unit space \(X\subseteq G\) is
  \emph{topologically free} if there is no non-empty open subset
  \(U\subseteq G\setminus X\) with \(r|_U = s|_U\).
\end{definition}

\begin{remark}
  Topological freeness is weaker than similar popular conditions
  like being effective or topologically principal (see
  \cite{Kwasniewski-Meyer:Essential}*{Section~2.4}).
\end{remark}

Let~\(G\) be an étale groupoid with locally compact and Hausdorff unit
space~\(X\).  A \emph{Fell bundle over the groupoid~\(G\)} is an
upper-semicontinous bundle \(\A=(A_\gamma)_{\gamma\in G}\) of Banach
spaces equipped with a continuous involution \(^*\colon \A\to \A\)
and a continuous partially defined multiplication
\({\cdot}\colon \setgiven{(a,b)\in \A\times \A} {a\in A_{\gamma_1},\
  b \in A_{\gamma_2},\ (\gamma_1,\gamma_2)\in G^{(2)}} \to \A\) that
satisfy a number of natural properties (see
\cites{BussExel:Fell.Bundle.and.Twisted.Groupoids,
  Kwasniewski-Meyer:Essential} for details).  Then
\(A\defeq\Cont_0(X,\A|_X)\) is a \(C_0(X)\)-\(\Cst\)-algebra and the
space of \(\Cont_0\)\nb-sections \(\Cont_0(U,\A|_U)\) for a
bisection \(U\in \Bis( G)\) becomes a Hilbert \(A\)\nb-bimodule.
These Hilbert bimodules form a Fell bundle over the inverse
semigroup \(\Bis(G)\).  This Fell bundle is saturated when~\(\A\) is
saturated.  In general, we change the construction by letting~\(S\)
be the family of all Hilbert subbimodules of \(\Cont_0(U,\A|_U)\)
for all \(U\in \Bis(G)\), cf.  \cite{Kwasniewski-Meyer:Essential}*{Lemma 7.3}.  This defines a saturated Fell bundle by
construction, which is the same as an action by Hilbert bimodules
on~\(A\).  And \cite{Kwasniewski-Meyer:Essential}*{Proposition 7.6
  and~7.9 and Definition~7.12} give natural isomorphisms
\begin{equation}
  \label{eq:groupoid_isomorphisms}
  \Cst(G,\A)\cong A\rtimes S,\qquad
  \Cst_\red(G,\A)\cong A\rtimes_{\red} S,\qquad
  \Cst_\ess(G,\A)\cong A\rtimes_\ess S
\end{equation}
between the corresponding full, reduced and essential
\(\Cst\)\nb-algebras.

\begin{definition}
  We will call a Fell bundle \(\A=(A_\gamma)_{\gamma\in G}\) over an
  étale groupoid~\(G\) an \emph{action of the groupoid}~\(G\) on
  \(\Cont_0(X,\A|_X)\).  An \emph{exotic crossed product} for~\(\A\)
  is a \(\Cst\)\nb-algebra \(B\) with surjective
  \Star{}homomorphisms \(\Cst(G,\A) \onto B\onto \Cst_\ess(G,\A)\)
  that compose to the quotient map
  \(\Cst(G,\A)\onto \Cst_\ess(G,\A)\).
\end{definition}

\begin{definition}
  \label{def:groupoid_action_aperiodic}
  A \emph{Fell bundle \(\A=(A_\gamma)_{\gamma\in G}\) over an étale
    groupoid~\(G\)} is \emph{aperiodic} if the corresponding inverse
  semigroup action described above is aperiodic.
\end{definition}

Proposition~\ref{pro:aperiodic_isg_action} implies that the
inclusion of \(\Cont_0(X,\A|_X)\) into any exotic crossed product
for a Fell bundle~\(\A\) over~\(G\) is aperiodic if and only if the
action \(\A\) is aperiodic.  All results about actions of étale
groupoids below are proven by reducing the statement to inverse
semigroup actions as above.

\section{Aperiodic inclusions and pseudo-expectations}
\label{sec:aperiodic_pseudo-expectation}

In this section we improve
\cite{Kwasniewski-Meyer:Essential}*{Theorem~5.28}, the main theorem
about general aperiodic inclusions
in~\cite{Kwasniewski-Meyer:Essential}, by proving some results about
the inclusion into the injective hull of a \(\Cst\)\nb-algebra.
Namely, any aperiodic inclusion admits a unique pseudo-expectation,
and all pseudo-expectations satisfy a technical condition that is
needed in \cite{Kwasniewski-Meyer:Essential}*{Theorem~5.28}.  Thus
the conclusions of \cite{Kwasniewski-Meyer:Essential}*{Theorem~5.28}
hold for all aperiodic inclusions together with their unique
pseudo-expectation.

\begin{lemma}
  \label{lem:largest_aperiodic_subbimodule}
  Any normed \(A\)\nb-bimodule~\(X\) contains a largest aperiodic
  \(A\)\nb-subbimodule.  For any \(\Cst\)\nb-inclusion
  \(A\subseteq B\) the \(\Cst\)-algebra~\(B\) contains a largest
  two-sided ideal which is aperiodic as an \(A\)\nb-bimodule.
\end{lemma}

\begin{proof}
  This follows from
  \cite{Kwasniewski-Meyer:Aperiodicity}*{Lemma~4.2} or
  \cite{Kwasniewski-Meyer:Essential}*{Lemma~5.12} because the closed
  linear span of a family of aperiodic subbimodules is again
  aperiodic.
\end{proof}

The largest aperiodic ideal \(\Null \subseteq B\) for a
\(\Cst\)\nb-inclusion \(A\subseteq B\) plays an important role.
By~\cite{Kwasniewski-Meyer:Essential}*{Theorem~5.17}, if the
inclusion \(A\subseteq B\) is aperiodic, then~\(\Null\) is the
unique ideal for
which \(A\cap \Null = 0\) and the induced inclusion
\(A \to B/\Null\) detects ideals in the following sense:

\begin{definition}
  Let~\(A\) be a \(\Cst\)\nb-subalgebra of a
  \(\Cst\)\nb-algebra~\(B\).  We say that~\(A\) \emph{detects
    ideals} in~\(B\) if \(J\cap A=0\) implies \(J=0\) for every
  ideal~\(J\) in~\(B\).  (Some authors then say that the
  \(\Cst\)\nb-inclusion \(A\subseteq B\) is essential
  (see~\cite{Pitts-Zarikian:Unique_pseudoexpectation}) or that it
  has the ideal intersection property.)
\end{definition}

Assume for a moment that our aperiodic inclusion \(A\subseteq B\)
carries a conditional expectation \(E\colon B\to A\).
Then~\(\Null\) is equal to the largest two-sided ideal contained in
the kernel of~\(E\) (see~\cite{Kwasniewski-Meyer:Essential}*{Theorem
  5.28}).  This applies, in particular, to all crossed products for
group actions.  In this case, the quotient \((A\rtimes G)/\Null\) is
the reduced crossed product \(A\rtimes_\red G\) because~\(E\)
induces a faithful conditional expectation \(A\rtimes_\red G\to A\).
A conditional expectation also exists for the inclusion
\(\Cont_0(X,\A|_X) \subseteq \Cst(G,\A)\) if~\(G\) is a Hausdorff,
étale, locally compact groupoid.  Once again, \(\Cst(G,\A)/\Null\)
is the reduced crossed product \(\Cst_\red(G,\A)\).  Thus
\(A\subseteq A\rtimes_\red G\) and
\(\Cont_0(X,\A|_X) \subseteq \Cst_\red(G,\A)\) detect ideals if the
underlying action is aperiodic.  The situation is different,
however, for general inverse semigroup actions or Fell bundles over
non-Hausdorff groupoids.  They do not admit a genuine conditional
expectation.  The way out is to consider ``generalised''
expectations, which take values in a larger \(\Cst\)\nb-algebra
\(\tilde{A}\supseteq A\):

\begin{definition}
  A \emph{generalised expectation} for a \(\Cst\)\nb-inclusion
  \(A\subseteq B\) consists of another \(\Cst\)\nb-inclusion
  \(A\subseteq \tilde{A}\) and a completely positive, contractive
  map \(E\colon B \to \tilde{A}\) that restricts to the identity map
  on~\(A\).
\end{definition}

Any generalised expectation is an \(A\)\nb-bimodule map by
\cite{Kwasniewski-Meyer:Essential}*{Lemma~3.2}.

The identity map on~\(B\) is a generalised expectation for any
\(\Cst\)\nb-inclusion, and it cannot tell us anything interesting.
Therefore, an extra condition on a generalised expectation is
needed.  The main theorem for general \(\Cst\)\nb-inclusions
in~\cite{Kwasniewski-Meyer:Essential} requires a generalised
expectation which is ``supportive'' (see Definition
\ref{def:supportive_expectation} below).  And it asserts that~\(A\)
supports~\(B/\Null\) in the following sense:

\begin{definition}
  Let~\(B^+\) be the set of positive elements in~\(B\).  We equip
  \(B^+\setminus\{0\}\) with the \emph{Cuntz preorder}~\(\precsim\)
  introduced in~\cite{Cuntz:Dimension_functions}: for
  \(a, b\in B^+\setminus\{0\}\), we write \(a \precsim_B b\) and
  say that~\(a\) \emph{supports}~\(b\) (in~\(B\)) if, for every
  \(\varepsilon>0\), there is \(x \in B\) with
  \(\norm{a-x^* b x} <\varepsilon\).  We say that~\(A\)
  \emph{supports}~\(B\) if for every \(b\in B^+\setminus\{0\}\)
  there is \(a\in A^+\setminus\{0\}\) with \(a \precsim_B b\).
\end{definition}

\begin{definition}
  A \emph{pseudo-expectation} for a \(\Cst\)\nb-inclusion
  \(A\subseteq B\) is a generalised expectation
  \(E\colon B\to I(A)\) taking values in Hamana's injective hull
  (see~\cite{Hamana:Injective-Envelope-Cstar}).
\end{definition}

The injectivity of~\(I(A)\) implies that any \(\Cst\)\nb-inclusion
has at least one pseudo-expectation.  Having a unique
pseudo-expectation is an important structural property, which has
been advocated, in particular, by Pitts (see \cites{Pitts:Regular_I,
  Pitts:Regular_II, Pitts-Zarikian:Unique_pseudoexpectation,
  Zarikian:Unique_expectations}).

Now we explain how we are going to improve the main theorem about
general aperiodic inclusions in~\cite{Kwasniewski-Meyer:Essential}.
We show, first, that any aperiodic inclusion has a unique
pseudo-expectation~\(E\); secondly, that~\(E\) is supportive; and,
thirdly, that the largest aperiodic bimodule is \(\ker E\).  When we
put this information into
\cite{Kwasniewski-Meyer:Essential}*{Theorem~5.28}, then we get the
following theorem:

\begin{theorem}
  \label{thm:aperiodicity_implies_unique_expectation}
  Let \(A\subseteq B\) be an aperiodic \(\Cst\)\nb-inclusion.  Then
  there is exactly one pseudo-expectation \(E\colon B\to I(A)\) and \(\ker E\) is  the largest aperiodic \(A\)-subbimodule in~\(B\).
  Let~\(\Null\) be the largest two-sided ideal contained in
  \(\ker E\).  Then
  \begin{enumerate}
  \item \label{enu:aperiodic_consequences1}%
    for every \(b\in B^+\) with \(b\notin \Null\), there is
    \(a\in A^+\setminus\{0\}\) with \(a \precsim_B b\); in
    particular, \(A\) supports \(B/\Null\);
  %\item \label{enu:aperiodic_consequences3}%

  \item \label{enu:aperiodic_consequences4}%
    if \(J\subseteq B\) is an ideal with \(J\cap A = 0\), then
    \(J\subseteq \Null\); in particular, \(A\) detects ideals
    in~\(B/\Null\), and \(B/\Null\) is the unique quotient of~\(B\)
    with this property;
  \item \label{enu:aperiodic_consequences5}%
    \(B\) is simple if and only if \(\Null=0\) and
    \(A\subseteq B I B\) for any non-zero ideal~\(I\) in~\(A\);
  \item \label{enu:aperiodic_consequences6}%
    if~\(B\) is simple, then~\(B\) is purely infinite if and only if
    all elements of \(A^+\setminus\{0\}\) are infinite in~\(B\).
 % \item \label{enu:aperiodic_consequences7}% .
  \end{enumerate}
\end{theorem}

\begin{remark}
  \label{rem:ess_defined_vs_pseudo-expectation}
  The essential crossed products
  in~\cite{Kwasniewski-Meyer:Essential} are defined using a
  generalised expectation into the local multiplier algebra,
  \(E\colon B\to \Locmult(A)\).  There is a canonical embedding
  \(\iota\colon \Locmult(A) \hookrightarrow I(A)\) by
  \cite{Frank:Injective_local_multiplier}*{Theorem~1}.  Thus
  generalised expectations into \(\Locmult(A)\) become
  pseudo-expectations as well, and this change of view point affects
  neither the kernel \(\ker E\) nor the largest two-sided ideal
  contained in \(\ker E\).  As a result, if \(B = A\rtimes S\) for
  an inverse semigroup action, then \(B/\Null\) is the essential
  crossed product~\(A\rtimes_\ess S\) defined
  in~\cite{Kwasniewski-Meyer:Essential}.  If \(B = \Cst(G,\A)\)
  is a section \(\Cst\)\nb-algebra for a Fell bundle over an étale
  groupoid~\(G\), then \(B/\Null\) is the essential section
  \(\Cst\)\nb-algebra~\(\Cst_\ess(G,\A)\) as
  in~\cite{Kwasniewski-Meyer:Essential}.  Thus
  Theorem~\ref{thm:aperiodicity_implies_unique_expectation} contains
  criteria for \(\Cst\)\nb-algebras of the form \(A\rtimes_\ess S\)
  and \(\Cst_\ess(G,\A)\) to be simple or purely infinite.  In fact,
  these criteria are already proven
  in~\cite{Kwasniewski-Meyer:Essential}, using the generalised
  expectation \(E\colon B\to \Locmult(A)\).
\end{remark}

The proof of
Theorem~\ref{thm:aperiodicity_implies_unique_expectation} will
occupy the rest of this section.  The following concept is crucial
for the proof:

\begin{definition}
  A \(\Cst\)\nb-inclusion \(A\subseteq \tilde{A}\) is called
  \emph{anti-aperiodic} if there are no non-zero aperiodic
  \(A\)\nb-subbimodules of~\(\tilde{A}\).
\end{definition}

The second and more difficult part of the proof will show that the
inclusion \(A \hookrightarrow I(A)\) is anti-aperiodic.  The first
part consists of several rather easy results about generalised
expectations \(B\to\tilde{A}\) when the inclusion
\(A\subseteq \tilde{A}\) is anti-aperiodic.

\begin{proposition}
  \label{pro:two_gen_expectations}
  Let \(B\supseteq A\subseteq \tilde{A}\) be \(\Cst\)\nb-inclusions
  such that \(A\subseteq B\) is aperiodic and
  \(A\subseteq \tilde{A}\) is anti-aperiodic.  Then there is at most
  one generalised expectation \(B\to\tilde{A}\).
\end{proposition}

\begin{proof}
  Let \(E_1,E_2\colon B \rightrightarrows \tilde{A}\) be two
  generalised expectations.  The map
  \(E_1-E_2\colon B\to \tilde{A}\) is an \(A\)\nb-bimodule map that
  vanishes on~\(A\).  Thus it descends to a bounded
  \(A\)\nb-bimodule map \(B/A \to \tilde{A}\).  Since~\(B/A\) is an
  aperiodic \(A\)\nb-bimodule, the range of the map \(E_1 - E_2\) is
  an aperiodic \(A\)\nb-subbimodule of~\(\tilde{A}\) by
  \cite{Kwasniewski-Meyer:Essential}*{Lemma~5.12}.  Since
  \(A\subseteq \tilde{A}\) is anti-aperiodic, the range must
  be~\(0\).  So \(E_1 = E_2\).
\end{proof}

\begin{definition}[\cite{Kwasniewski-Meyer:Essential}*{Definition 5.19}]\label{def:supportive_expectation}
  A generalised expectation \(E\colon B\to \tilde{A} \supseteq A\)
  is called \emph{supportive} if no non-zero element of \(E(B^+)\)
  satisfies Kishimoto's condition, that is, if, for any \(b\in B^+\)
  with \(E(b)\neq0\), there are \(\delta>0\) and \(D\in\Her(A)\)
  such that \(\norm{xE(b) x}\ge \delta\) for all \(x\in D^+_1\).
\end{definition}

\begin{proposition}
  \label{pro:aper_bimodules_supportive}
  Let \(B\supseteq A\subseteq \tilde{A}\) be \(\Cst\)\nb-inclusions.
  If \(A\subseteq \tilde{A}\) is anti-aperiodic,
  then Kishimoto's condition fails for any non-zero positive element
  in~\(\tilde{A}\), and so any generalised expectation
  \(E\colon B \rightarrow \tilde{A}\) is supportive.
\end{proposition}

\begin{proof}
  Let \(0 \neq b\in \tilde{A}^+\).  Let
  \(A^1 \defeq A\oplus \C\cdot 1\) be the unital \(\Cst\)\nb-algebra
  generated by~\(A\) and extend the \(A\)\nb-bimodule structure
  on~\(\tilde{A}\) to an \(A^1\)\nb-bimodule structure.  By
  assumption, the \(A\)\nb-bimodule \(A^1 b^{1/2} A^1\) is not
  aperiodic.  The set of elements that satisfy Kishimoto's condition
  is a closed vector subspace by
  \cite{Kwasniewski-Meyer:Aperiodicity}*{Lemma~4.2}.  The unital
  \(\Cst\)\nb-algebra~\(A^1\) is spanned by unitaries.  Hence there
  must be unitaries \(u,v\in A^1\) such that Kishimoto's condition
  fails for \(u b^{1/2} v\).  That is, there are \(\delta>0\) and
  \(D\in \Her(A)\) such that \(\norm{x u b^{1/2} v x}\ge \delta\)
  for all \(x\in D^+_1\).  Let \(y\in (u^* D u)^+_1\).  Then
  \(y = u^* x u\) for some \(x\in D^+_1\) and hence
  \[
    \norm{y b y}
    = \norm{x u b u^* x}
    \ge \norm{x u b^{1/2} v x^2 v^* b^{1/2} u^* x}
    = \norm{x u b^{1/2} v x}^2
    \ge \delta^2.
  \]
  So Kishimoto's condition fails for~\(b\).
\end{proof}

\begin{lemma}
  \label{lem:ker_E_max_aperiodic}
  Let \(B\supseteq A\subseteq \tilde{A}\) be \(\Cst\)\nb-inclusions
  such that \(A\subseteq B\) is aperiodic and
  \(A\subseteq \tilde{A}\) is anti-aperiodic.  Let
  \(E\colon B\to\tilde{A}\) be a generalised expectation.  Then
  \(\ker E\subseteq B\) is the largest aperiodic \(A\)\nb-bimodule
  in~\(B\).
\end{lemma}

\begin{proof}
  We first show that \(\ker E\) is aperiodic.  Since
  \(E|_A= \Id_A\), the quotient map from \(\ker E\) to~\(B/A\) is
  injective.  We claim that it is a topological isomorphism.  Then
  \(\ker E\) inherits aperiodicity from~\(B/A\).  Let
  \(b\in \ker E\) and let \(a\in A\).  Then \(a = E(a) = E(a-b)\).
  Since~\(E\) is contractive, this implies
  \(\norm{a} \le \norm{a-b}\).  Then
  \(\norm{b-a} \ge \norm{b} - \norm{a} \ge \norm{b} - \norm{b-a}\).
  So \(\norm{b-a} \ge \norm{b}/2\).  This means that
  \(\norm{b}_{B/A} \ge \norm{b}_{\ker E}/2\).

  Now let \(X\subseteq B\) be any aperiodic \(A\)\nb-bimodule.  Then
  \(E(X) \subseteq \tilde{A}\) with the norm from~\(\tilde{A}\) is
  aperiodic by \cite{Kwasniewski-Meyer:Essential}*{Lemma~5.12}.
  Since \(A\subseteq \tilde{A}\) is anti-aperiodic, it follows that
  \(E(X)=0\).  So \(X\subseteq \ker E\).
\end{proof}

Now we begin to prove that the inclusion \(A\subseteq I(A)\) is
anti-aperiodic.

\begin{lemma}
  \label{lem:isometric_inclusion}
  Let \(A\subseteq \tilde{A}\) be a \(\Cst\)\nb-inclusion and let
  \(X\subseteq \tilde{A}\) be an aperiodic, closed
  \(A\)\nb-subbimodule.  The restriction of the quotient map
  \(A\to \tilde{A}/X\) is isometric.
\end{lemma}

\begin{proof}
  If \(X\subseteq \tilde{A}\) is aperiodic, then so is~\(X^*\)
  because \(\norm{a b^* a} = \norm{a^* b a^*}\).  We get a stronger
  statement if we replace~\(X\) by the closed linear span of \(X\)
  and~\(X^*\), and the latter is aperiodic as well by
  Lemma~\ref{lem:largest_aperiodic_subbimodule}.  Therefore, we may
  assume without loss of generality that \(X = X^*\).

  Let \(a\in A\), \(b\in X\) and \(\varepsilon >0\).
  There is \(D\in \Her(A)\) such that
  \(\norm{x a^* a x}\ge \norm{a^*a} -\varepsilon/2\) for all
  \(x\in D^+_1\) (see
  \cite{Kwasniewski-Meyer:Aperiodicity}*{Lemma~2.9}).  Since
  \(a^* b + b^* a\in X\), there is \(x\in D^+_1\)
  with \(\norm{x(a^* b+b^* a)x}< \varepsilon/2\).  Then
  \begin{align*}
    \norm{a+b}^2
    &= \norm{a^*a+a^*b+b^*a+b^*b}
    \ge \norm{x(a^*a+a^*b+b^*a+b^*b)x}
    \\&\ge \norm{xa^*ax+xb^*bx} - \norm{x(a^*b+b^*a)x}
    > \norm{xa^*ax} -\varepsilon/2
    \\&\ge \norm{a^*a}-\varepsilon
    = \norm{a}^2-\varepsilon.
  \end{align*}
  Since \(\varepsilon>0\) is arbitrary, this implies
  \(\norm{a+b} \ge \norm{a}\).  That is, the quotient
  norm on \(\tilde{A}/X\) restricts to the usual
  norm on~\(A\).
\end{proof}

\begin{lemma}
  \label{lem:aperiodic_to_matrices}
  Let~\(A\) be a \(\Cst\)\nb-algebra and let \(n\ge1\).  Let~\(X\)
  be an aperiodic normed \(A\)\nb-bimodule.  Then \(\Mat_n(X)\) is
  an aperiodic \(\Mat_n(A)\)\nb-bimodule.
\end{lemma}

\begin{proof}
  Let \(x = (x_{i,j})_{1 \le i,j\le n}\in \Mat_n(X)\),
  \(D\in \Her(\Mat_n(A))\) and \(\varepsilon>0\).  We are going to
  check Kishimoto's condition for this data.  Equip~\(A^n\) with the
  standard right Hilbert \(A\)\nb-module structure.  Then
  \(\Comp(A^n)\cong \Mat_n(A)\), and~\(A^n\) is an
  \(\Mat_n(A)\)-\(A\)-equivalence bimodule.  Then
  \(D\cdot A^n\subseteq A^n\) is a right \(A\)\nb-submodule and
  \(D\cong \Comp(D\cdot A^n)\).  So \(D\cdot A^n\neq0\) and there is
  \(\eta= (\eta_k)_{1 \le k \le n}\in D\cdot A^n \subseteq A^n\)
  with \(\norm{\eta}=1\).  Then
  \cite{Kwasniewski-Meyer:Aperiodicity}*{Lemma~2.9} applied to
  \(\abs{\eta} \defeq \langle\eta,\eta \rangle_A^{1/2}\in A\) gives
  \(D_0\in \Her(\overline{\abs{\eta} A\abs{\eta}})\) with
  \(\norm{\abs{\eta} b} \ge (1-\varepsilon) \norm{b}\) for all
  \(b\in D_0\).  If \(b\in D_0\), then
  \begin{equation}
    \label{eq:lower_bound_eta_b}
    \norm{\eta b}^2
    = \norm{\braket{\eta b}{\eta b}_A}
    = \norm{b^* \abs{\eta}^2 b}
    = \norm{\abs{\eta} b}^2
    \ge (1-\varepsilon)^2\norm{b}^2.
  \end{equation}
  Put \(w_{i,j}\defeq \eta_i^*\cdot x_{i,j}\cdot \eta_j \in X\).
  Since~\(X\) is an aperiodic \(A\)\nb-bimodule, so is~\(X^{n^2}\)
  by \cite{Kwasniewski-Meyer:Aperiodicity}*{Lemma~4.2}.  Then there
  is \(b\in (D_0)^+_1\) with
  \(\norm{b w_{i,j} b}<\varepsilon (1-\varepsilon)^2/n^4\) for all
  \(1 \le i,j \le n\).  Let \(b_0\defeq b/\norm{\eta b}\in D_0^+\),
  so that \(\norm{\eta b_0} = 1\).
  Using~\eqref{eq:lower_bound_eta_b}, we get
  \(\norm{b_0 w_{i,j} b_0}<\varepsilon/n^4\).  The rank-one operator
  \(\ket{\eta b_0}\bra{\eta b_0}\) belongs to
  \(\Comp(D\cdot A^n)^+_1\) because
  \(\norm{\ket{\eta b_0}\bra{\eta b_0}} = \norm{\eta b_0}^2 = 1\)
  and \(\eta b_0 \in D\cdot A^n\).  The isomorphism
  \(\Comp(D\cdot A^n)\cong D\) maps it to an element \(a\in D^+_1\).
  We claim that \(\norm{a x a}<\varepsilon\).  For this computation,
  write \(a \in D \subseteq \Mat_n(A)\) as a matrix
  \((a_{k,l})_{1\le k,l\le n}\) with
  \(a_{k,l} = \eta_k b_0^2 \eta_l^*\).  So~\(a x a\) is the matrix
  with \(k,l\)-entry
  \(\sum_{i,j=1}^n \eta_k b_0^2\eta_i^* x_{i,j} \eta_j
  b_0^2\eta_l^*\).  We estimate
  \begin{align*}
    \norm{a x a}
    &\le \sum_{k,l,i,j=1}^n
      {}\norm{\eta_k b_0^2 \eta_i^* x_{i,j} \eta_j b_0^2\eta_l^*}
      = \sum_{k,l,i,j=1}^n
      {}\norm{(\eta_k b_0) b_0w_{i,j} b_0 (\eta_l b_0)^*}
    \\&\le \sum_{k,l,i,j=1}^n {}\norm{b_0 w_{i,j} b_0}
    < \varepsilon.
  \end{align*}
  Thus \(x\in \Mat_n(X)\) satisfies Kishimoto's condition.
\end{proof}

\begin{lemma}
  \label{lem:completely_isometric_inclusion}
  Let \(A\subseteq \tilde{A}\) be a \(\Cst\)\nb-inclusion and
  \(X\subseteq \tilde{A}\) an aperiodic \(A\)\nb-subbimodule.  The
  restriction of the quotient map \(A\to \tilde{A}/X\) is completely
  isometric.
\end{lemma}

\begin{proof}
  Let \(n\ge1\).  By definition,
  \(\Mat_n\bigl(\tilde{A}/X\bigr)\cong \Mat_n(\tilde{A})/\Mat_n(X)\)
  with the quotient semi-norm.  By
  Lemma~\ref{lem:aperiodic_to_matrices},
  \(\Mat_n(X) \subseteq \Mat_n(\tilde{A})\) is aperiodic.  Then
  Lemma~\ref{lem:isometric_inclusion} shows that the map
  \(\Mat_n(A)\to \Mat_n(\tilde{A}) \bigm/ \Mat_n(X)\) is isometric.
\end{proof}

\begin{proposition}
  \label{pro:Ape_IA}
  For any \(\Cst\)\nb-algebra, the inclusion \(A\subseteq I(A)\) is
  anti-aperiodic.  That is, the injective hull \(I(A)\) contains no
  non-zero aperiodic \(A\)\nb-subbimodule.
\end{proposition}

\begin{proof}
  Let \(\tilde{A}=I(A)\) and let~\(X\) be an aperiodic
  \(A\)\nb-subbimodule of~\(I(A)\).
  Lemma~\ref{lem:completely_isometric_inclusion} says that the map
  \(A\to I(A)/X\) is completely isometric.  Since \(I(A)\) is
  injective, the inclusion \(A \hookrightarrow I(A)\) extends to a
  completely contractive map \(h\colon I(A)/X\to I(A)\).  Hence the
  composite map \(I(A) \onto I(A)/X \to I(A)\) is completely
  contractive and it restricts to the identity map on~\(A\).  The
  rigidity of the injective envelope in
  \cite{Paulsen:Completely_bounded}*{Corollary~15.7} implies that
  any such map is equal to the identity map on \(I(A)\).  This can
  only happen if \(X=0\).
\end{proof}

\begin{proof}[Proof of
  Theorem~\textup{\ref{thm:aperiodicity_implies_unique_expectation}}]
  A pseudo-expectation \(E\colon B\to I(A)\) exists because~\(I(A)\)
  is injective.  By  Proposition~\ref{pro:Ape_IA}  the
  inclusion \(A\subseteq I(A)\) is anti-aperiodic.  Then
  Proposition~\ref{pro:two_gen_expectations} shows that~\(E\) is the
  unique pseudo-expectation \(B\to I(A)\).
  Lemma~\ref{lem:ker_E_max_aperiodic} shows that \(\ker E\) is the
  largest aperiodic bimodule in~\(A\).
  By Proposition~\ref{pro:aper_bimodules_supportive},~\(E\)
  is supportive.  Then the remaining assertions follow from
  \cite{Kwasniewski-Meyer:Essential}*{Theorem~5.28}.
\end{proof}

\begin{remark}
  Is the inclusion \(A \hookrightarrow I(A)\) aperiodic?  While we
  do not have a proof for this, there is some positive evidence.
  First, if~\(A\) is commutative, then \(I(A) = \Locmult(A)\) (see
  \cite{Frank:Injective_local_multiplier}*{Theorem~1}); and the
  inclusion \(A\subseteq \Locmult(A)\) is shown to be aperiodic
  in~\cite{Kwasniewski-Meyer:Essential}.  Secondly, the inclusion
  \(A \hookrightarrow I(A)\) has a unique pseudo-expectation:
  it must be the identity map by the rigidity of \(I(A)\).
\end{remark}

\section{Topological freeness implies aperiodicity}
\label{sec:top_free_aperiodic}
		
In this section, we show that topologically free actions are
aperiodic.  This applies to actions of groups, actions of inverse
semigroups by Hilbert bimodules, or Fell bundles over étale locally
compact groupoids.  The proof reduces to a statement about Hilbert
bimodules.  For Hilbert bimodules over separable
\(\Cst\)\nb-algebras, this is already shown
in~\cite{Kwasniewski-Meyer:Aperiodicity}, where the proof is based
on a statement about automorphisms shown by Olesen--Pedersen
in~\cite{Olesen-Pedersen:Applications_Connes_3}.  Here we give a
direct proof, which applies to arbitrary \(\Cst\)\nb-algebras.

Let \(A\) be a \(\Cst\)\nb-algebra and let~\(\dual{A}\) be its
spectrum.  So~\(\dual{A}\) is the set of unitary equivalence classes
of irreducible representations of~\(A\), equipped with the topology
where the open subsets are \(\dual{J} \subseteq \dual{A}\) for all
ideals \(J\) in \(A\).  Any automorphism \(\alpha\colon A\to A\)
induces a homeomorphism \(\dual{\alpha}\colon\dual{A}\to \dual{A}\)
by \(\dual{\alpha}[\varrho]=[\varrho \circ \alpha]\) for
\([\varrho]\in \dual{A}\).  The automorphism~\(\alpha\) is called
\emph{topologically non-trivial} if the set of
\([\varrho]\in\widehat{A}\) with
\(\widehat{\alpha}[\varrho] \neq [\varrho]\) is dense or,
equivalently,
\(\setgiven{[\varrho]\in \widehat{A}}{[\varrho\circ\alpha] =
  [\varrho]}\) has empty interior in~\(\widehat{A}\).  More
explicitly, for every non-zero ideal~\(I\) in~\(A\), there is an
irreducible representation~\(\varrho\) of~\(I\) such that
\(\varrho\circ \alpha\) and~\(\varrho\) are not unitarily
equivalent.

\begin{definition}
  We call a group action \(\alpha\colon G\to \Aut(A)\)
  \emph{topologically free} if, for each \(g\in G\setminus\{1\}\),
  the automorphism~\(\alpha_g\) is topologically non-trivial.
\end{definition}

\begin{remark}
  \label{rem:Archbold_Spielberg}
  The condition above differs slightly from the one
  in~\cite{Archbold-Spielberg:Topologically_free}, which is used in
  a number of papers
  (including~\cite{Kwasniewski-Meyer:Aperiodicity}) and that
  requires that the union
  \(\bigcup_{k=1}^n \setgiven{[\varrho]\in
    \widehat{A}}{[\varrho\circ\alpha_{g_{k}}] = [\varrho]}\) for
  \(g_1,\dots,g_n\in G\setminus\{1\}\) has empty interior
  in~\(\widehat{A}\).  The two conditions are equivalent when~\(A\)
  contains an essential ideal which is separable or of Type~I (see
  \cite{Kwasniewski-Meyer:Aperiodicity}*{Proposition~9.7}).  In
  general, the results in this article show that our (formally)
  weaker condition implies even stronger results than those
  in~\cite{Archbold-Spielberg:Topologically_free}.
\end{remark}

The dual action of~\(G\) on~\(\dual{A}\) has a transformation
groupoid \(\dual{A}\rtimes G\).  This is an étale topological
groupoid, where the unit space is~\(\dual{A}\), morphisms from
\([\pi]\) to~\([\varrho]\) correspond to the elements \(g\in G\)
such that \(g\cdot [\pi]=[\varrho]\), and the topology on the arrow
space \(G\times \dual{A}\) is the product topology.  The group
action~\(\alpha\) on~\(A\) is topologically free if and only if the
groupoid \(\dual{A}\rtimes G\) is topologically free as in
Definition~\ref{def:topologically_free_groupoid}.  Similar dual
groupoids are defined for actions of inverse semigroups by Hilbert
bimodules and for Fell bundles over étale groupoids.  We will use
this to define topologically free actions of inverse semigroups and
étale groupoids.

To generalise the dual action to Fell bundles, we must explain how a
Hilbert \(A\)\nb-bimodule~\(\Hilm\) induces a partial homeomorphism
of~\(\dual{A}\).  Let \(\s(\Hilm)\) and \(\rg(\Hilm)\) be the closed
ideals in~\(A\) that are generated by the right and the left inner
products.  In symbols, \(\s(\Hilm) = \overline{\mathrm{span}}
\,\braket{\Hilm}{\Hilm}\) and
\(\rg(\Hilm) = \overline{\text{span}} \,\BRAKET{\Hilm}{\Hilm}\).  Let
\(\varrho\colon A\to\Bound(\Hils)\) be an irreducible
representation.  The tensor product \(\Hilm \otimes_\varrho \Hils\)
is non-zero if and only if~\([\varrho]\) belongs
to~\(\dual{\s(\Hilm)}\).  Then the left multiplication action
of~\(A\) on \(\Hilm \otimes_\varrho \Hils\) is an irreducible
representation that belongs to~\(\dual{\rg(\Hilm)}\).  The unitary
equivalence class of the representation
\(\Hilm \otimes_\varrho \Hils\) depends only on the class
of~\(\varrho\), and the map
\([\varrho] \mapsto [\Hilm \otimes_\varrho \Hils]\) is a
homeomorphism
\(\dual{\Hilm}\colon \dual{\s(\Hilm)} \to \dual{\rg(\Hilm)}\).

\begin{definition}[\cite{Kwasniewski-Meyer:Aperiodicity}*{Definition~2.13}]
  \label{def:top_non-trivial}
  A Hilbert \(A\)\nb-bimodule~\(\Hilm\) over a
  \(\Cst\)\nb-algebra~\(A\) is \emph{topologically non-trivial} if
  for each ideal \(J\idealin A\) there is \([\varrho]\in\dual{J}\)
  with \(\dual{\Hilm}[\varrho]\neq[\varrho]\).
\end{definition}
%We defined in  \cite{Kwasniewski-Meyer:Essential} topologically free inverse semigroup actions in terms of the associated
%dual groupoid.  In view of By Proposition~\ref{pro:aperiodic_isg_action}, the action~\(\Hilm\)
%is aperiodic if and only if the Hilbert bimodules
\begin{definition}[\cite{Kwasniewski-Meyer:Essential}*{Definitions~2.14, 2.20}]
  Let \(\Hilm=(\Hilm_t,\mu_{t,u})_{t,u\in S}\) be an action of an
  inverse semigroup  on a \(\Cst\)\nb-algebra~\(A\) by
  Hilbert bimodules.  Then each~\(\Hilm_t\) for \(t\in S\) defines a
  partial homeomorphism~\(\dual{\Hilm_t}\) of~\(\dual{A}\).
	These partial homeomorphisms form an  action of the inverse semigroup~\(S\)
  on~\(\dual{A}\).  This  action has an étale transformation
  groupoid, which is the \emph{dual groupoid} of the
  action~\(\Hilm\).  The action~\(\Hilm\) is called
  \emph{topologically free} if this dual groupoid is topologically
  free.
\end{definition}

By Proposition~\ref{pro:aperiodic_isg_action}, the action~\(\Hilm\)
is aperiodic if and only if the Hilbert bimodules
\(\Hilm_t\cdot I_{1,t}^\bot\) are aperiodic for all \(t\in S\).
There is a similar criterion for topological freeness:

\begin{lemma}
  \label{lem:topological_free_non_trivial}
  An inverse semigroup action \(\Hilm=(\Hilm_t)_{t\in S}\) is
  topologically free if and only if the Hilbert \(A\)\nb-bimodules
  \(\Hilm_t\cdot I_{t,1}^\bot\) for \(t\in S\) are topologically
  non-trivial.
\end{lemma}

\begin{proof}
  This lemma is a part of
  \cite{Kwasniewski-Meyer:Essential}*{Theorem~6.13}.  We recall the
  relevant part of the proof.  The dual groupoid
  \(\dual{A}\rtimes S\) is covered by bisections associated to
  \(t\in S\).  The action of the bisection for \(t\in S\) is the
  dual action of the Hilbert bimodule~\(\Hilm_t\).  Removing the
  closure of the unit bisection replaces the partial homeomorphism
  associated to~\(\Hilm_t\) by the partial homeomorphism associated
  to \(\Hilm_t\cdot I_{1,t}^\bot\).  Thus \(\dual{A}\rtimes S\) is
  topologically free if and only if
  each~\(\Hilm_t\cdot I_{1,t}^\bot\) is topologically non-trivial.
\end{proof}

\begin{example}
  \label{exa:transformation_groupoid_Fell}
  Let~\(\A\) be a Fell bundle over an étale groupoid~\(G\) with
  locally compact Hausdorff object space~\(X\).  Then~\(G\) acts
  naturally on the spectrum~\(\dual{A}\) of the \(\Cst\)\nb-algebra
  \(A\defeq \Cont_0(X, \A)\).  Every irreducible representation
  of~\(A\) factors through the evaluation map \(A \to A_x\) for some
  \(x\in X\).  This defines a continuous map
  \(\psi\colon \dual{A} \to X\).  This is the anchor map of a
  \(G\)\nb-action on~\(\dual{A}\).  Namely, \(\gamma\in G\) acts by
  the partial homeomorphisms
  \(\psi_\gamma\colon\dual{A_{\s(\gamma)}} \to
  \dual{A_{\rg(\gamma)}}\) induced by the Hilbert
  \(A_{\rg(\gamma)},A_{\s(\gamma)}\)-bimodule \(A_\gamma\) (see, for
  instance,
  \cite{Ionescu-Williams:Remarks_ideal_structure}*{Section~2}).  The
  corresponding transformation groupoid is
  \(\dual{A}\rtimes G\defeq \setgiven{(\gamma,[\pi])\in G\times
    \dual{A}}{s(\gamma)=\psi([\pi])}\). The elements \((\eta,[\rho])\)
  and \((\gamma,[\pi])\) are composable if and only if
  \([\rho]=\psi_{\gamma}([\pi])\), and then their composite is
  \((\eta\gamma,[\pi])\).  The inverse is given by
  \((\gamma,[\pi])\mapsto (\gamma^{-1},\psi_{\gamma}([\pi]))\).  For
  the inverse semigroup action on~\(A\) used in the
  isomorphisms~\eqref{eq:groupoid_isomorphisms}, there is also a
  natural isomorphism between the transformation groupoids
  \[
    \dual{A}\rtimes G\cong \dual{A}\rtimes S
  \]
  (see the discussion before
  \cite{Kwasniewski-Meyer:Essential}*{Remark~7.4}).  Therefore, we
  will call \(\dual{A}\rtimes G\) the \emph{dual groupoid for the
    groupoid action} \(\A\), and we will say that the \emph{groupoid
    action~\(\A\) is topologically free} if this dual groupoid is
  topologically free.
\end{example}

The following theorem is the main result in this section:

\begin{theorem}
  \label{the:top_free_vs_aperiodic}
  Let~\(A\) be a \(\Cst\)\nb-algebra and let~\(\Hilm\) be a Hilbert
  \(A\)\nb-bimodule.  If~\(\Hilm\) is topologically non-trivial,
  then~\(\Hilm\) is aperiodic.
\end{theorem}

\begin{corollary}
  \label{cor:top_free_vs_aperiodic}
  Any topologically free action of  an inverse semigroup or
 an étale groupoid on a \(\Cst\)\nb-algebra is aperiodic.
\end{corollary}

\begin{proof}
  We have explained above why an inverse semigroup action
  \((\Hilm_t)_{t\in S}\) is topologically free or aperiodic if and
  only if the Hilbert bimodules \(\Hilm_t\cdot I_{1,t}^\bot\) are
  topologically non-trivial or aperiodic for all \(t\in S\),
  respectively.  Hence for such actions the assertion follows from
  Theorem \ref{the:top_free_vs_aperiodic}.  Actions of étale
  groupoids may be rewritten through inverse semigroup actions, and
  this preserves the properties of aperiodicity and topological
  freeness (see Definition~\ref{def:groupoid_action_aperiodic} and
  Example~\ref{exa:transformation_groupoid_Fell}).  Thus the
  statement for inverse semigroup actions implies the statements for
  actions of groups and étale groupoids.
\end{proof}

\begin{remark}
  \label{rem:aperiodic_vs_top_free}
  By \cite{Kwasniewski-Meyer:Aperiodicity}*{Theorem~8.1}, an
  aperiodic Hilbert \(A\)\nb-bimodule is topologically non-trivial
  provided~\(A\) contains an essential ideal which is separable or
  of Type~I.  In the separable case, this also follows from
  Theorem~\ref{the:almost_extension_aperiodic} below.

  The authors do not know an action that is aperiodic but not
  topologically free.  We speculate, however, that the
  counterexamples to Naimark's problem by Akemann and Weaver may
  give such examples.  Assuming a certain axiom in set theory,
  Akemann and Weaver~\cite{Akemann-Weaver:Naimark_problem} build a
  non-separable \(\Cst\)\nb-algebra~\(A\) such that~\(\dual{A}\) has
  only one point, although~\(A\) is not isomorphic to a
  \(\Cst\)\nb-algebra of compact operators on any Hilbert space.
  Then~\(A\) is simple.  By Kishimoto's Theorem, an automorphism
  of~\(A\) is aperiodic if and only if it is outer.  No automorphism
  of~\(A\) is topologically free because~\(\dual{A}\) has only one
  point.  We do not know whether~\(A\) admits outer automorphisms.
  If one exists, then it would give a Fell bundle over \(\Z\)
  or~\(\Z/p\) for a prime number~\(p\) that is aperiodic and not
  topologically free.
\end{remark}

The proof of Theorem~\ref{the:top_free_vs_aperiodic} will occupy the
rest of this section.  We shall use the concept of a net excising a
state from~\cite{Akemann-Anderson-Pedersen:Excising}.  We only need
pure states, and then an excising net may be arranged to have
additional useful properties.  To simplify notation, we will add
these extra properties to the definition.  We begin with some
preparation.  For a pure state~\(f\), let
\(\varrho\colon A \to \Bound(\Hils)\) be the GNS representation
for~\(f\) and let \(\xi\in\Hils\) be the cyclic vector with
\(\braket{\xi}{\varrho(a)\xi} = f(a)\) for all \(a\in A\).
Identify~\(\varrho\) with a direct summand in the universal
representation \(\upsilon\colon A \to \Bound(\Hils[K])\), and
identify the bidual~\(A''\) of~\(A\) with the bicommutant of~\(A\)
in \(\Bound(\Hils[K])\).  Since~\(\varrho\) is irreducible, the
orthogonal projection onto \(\C\cdot\xi\) is a minimal projection in
the bicommutant of~\(\varrho\) and hence in~\(A''\).  We denote this
minimal projection in~\(A''\) by~\(P_f\).

\begin{definition}
  Let~\(A\) be a \(\Cst\)\nb-algebra and \(f\colon A\to\C\) a pure
  state.  A decreasing net \((a_n)_{n\in N}\) in~\(A^+_1\) \emph{strongly
    excises}~\(f\) if \(f(a_n) = 1\) for  \(n\in N\),
  \[
    \lim {} \norm{a_n x a_n - f(x) a_n^2} = 0, 
  \]
  for  \(x\in A\), 
  and~\((a_n)_{n\in N}\)  converges towards~\(P_f\)
  in the strong topology in the bicommutant
  \(A''\subseteq \Bound(\Hils[K])\).
\end{definition}

\begin{proposition}
  \label{pro:excising_properties}
  Let~\(f\) be a pure state on a \(\Cst\)\nb-algebra~\(A\) and let
  \(D\in \Her(A)\) be such that~\(f|_D\) is also a state.  Then
  there is a net~\((a_n)_{n\in N}\) in~\(D\) that strongly
  excises~\(f\) as a state on~\(A\).
\end{proposition}

\begin{proof}
  Since~\(f|_D\) is a state, there is \(d\in D^+_1\) with
  \(f(d)=1\).  The construction in
  \cite{Akemann-Anderson-Pedersen:Excising}*{Proposition~2.2}
  applied to~\(d\) gives a decreasing net \((a_n)_{n\in N}\)
  in~\(D\) with \(\lim {} \norm{a_n x a_n - f(x) a_n^2} = 0\) for
  all \(x\in A\) and \(f(a_n) = 1\) for all \(n\in N\).  It is
  observed in~\cite{Akemann-Weaver:Naimark_problem} that this net
  converges towards~\(P_f\) in the strong topology.  Indeed, the
  net~\((a_n)_{n\in N}\) converges strongly because it is a
  decreasing net of positive elements.  The strong limit cannot
  be~\(0\) because \(f(a_n) = 1\) for all \(n\in N\).  Then the end
  of the proof of
  \cite{Akemann-Anderson-Pedersen:Excising}*{Proposition~2.3} shows
  that the strong limit of~\((a_n)_{n\in N}\) must be~\(P_f\).
\end{proof}

\begin{lemma}
  \label{lem:excising_norm_estimate}
  Let~\(A\) be a \(\Cst\)\nb-algebra, \(\Hilm\) a Hilbert
  \(A\)\nb-bimodule, and \(f\colon A\to\C\) a pure state.  Let
  \((a_n)_{n\in N}\) be a net that strongly excises~\(f\).  Let
  \(\varrho\colon A\to\Bound(\Hils)\) be the GNS representation
  of~\(f\).  Assume that \(\Hilm\otimes_A \varrho\) is not unitarily
  equivalent to~\(\varrho\).  Then \(\lim {}\norm{a_n x a_n} = 0\)
  for all \(x\in\Hilm\).
\end{lemma}

\begin{proof}
  Let \(\xi\in\Hils\) be the vector with
  \(\braket{\xi}{\varrho(a)\xi} = f(a)\) for all \(a\in A\).  For
  any \(x\in \Hils\), the vector
  \(x\otimes \xi \in \Hilm \otimes_\varrho \Hils\) defines a
  positive linear functional
  \[
    g\colon A\to\C,\qquad
    a\mapsto \braket[\big]{x\otimes \xi}{a\cdot x\otimes \xi}
    = \braket[\big]{\xi}{\varrho(\braket{x}{a\cdot x})\xi}
    = f\bigl(\braket{x}{a\cdot x}\bigr).
  \]
  By assumption, the left multiplication representation of~\(A\) on
  \(\Hilm\otimes_A \varrho\) is not unitarily equivalent
  to~\(\varrho\).  Then the extension of this representation
  to~\(A''\) maps the projection~\(P_f\) to~\(0\).  The strong
  convergence \(\lim a_n = P_f\) in the universal representation
  implies \(\lim g(a_n) = 0\).  Equivalently,
  \(\lim f(\braket{x}{a_n \cdot x}_A) = 0\).  This equation
  generalises the claim in the middle of the proof of
  \cite{Akemann-Weaver:Naimark_problem}*{Lemma~1}.  From this point
  on, we closely follow the proof of
  \cite{Akemann-Weaver:Naimark_problem}*{Lemma~1}.  Let
  \(\varepsilon>0\).  There is \(n_0 \in N\) with
  \(f(\braket{x}{a_{n_0} \cdot x}_A) < \varepsilon/2\).  Let
  \(y\defeq \braket{x}{a_{n_0} \cdot x}_A \in A\).  There is
  \(n_1 \in N\) with \(\norm{a_n (y- f(y)) a_n}<\varepsilon/2\) for
  \(n \ge n_1\) because~\((a_n)_{n\in N}\) excises~\(f\).  If
  \(n\ge n_0, n_1\), then \(a_n \le a_{n_0}\).  We estimate
  \begin{align*}
    \norm{a_n x a_n}^2
    &\le \norm{a_n^{1/2} x a_n}^2
      = \norm{\braket{a_n^{1/2} x a_n}{a_n^{1/2} x a_n}_A}
      = \norm{a_n\cdot \braket{x}{a_n x}_A\cdot a_n}
    \\
    &\le \norm{a_n\cdot \braket{x}{a_{n_0} x}_A\cdot a_n}
      \le \norm{a_n (y-f(y)) a_n} + f(y) \norm{a_n^2}
      \le \varepsilon.
  \end{align*}
  This finishes the proof.
\end{proof}

\begin{proof}[Proof of
  Theorem~\textup{\ref{the:top_free_vs_aperiodic}}]
  Let \(x\in\Hilm\) and \(D\in \Her(A)\).  We check Kishimoto's
  condition for this data.  Let \(ADA\subseteq A\) be the two-sided ideal
  generated by~\(D\).  If \(\s(\Hilm) \cap ADA = 0\), then
  \(\Hilm\cdot D = 0\).  Then \(\norm{x\cdot a}=0\) for any
  \(a\in D^+_1\).  This implies Kishimoto's condition for \(x\)
  and~\(D\).  So we may assume \(\s(\Hilm) \cap ADA \neq 0\).
  Since~\(\Hilm\) is topologically non-trivial, there is an
  irreducible representation \(\varrho\colon A \to \Bound(\Hils)\)
  which is non-zero on \(\s(\Hilm) \cap ADA\) and such that the
  (irreducible) representation of~\(A\) on
  \(\Hilm \otimes_\varrho \Hils\) by left multiplication is not
  unitarily equivalent to~\(\varrho\).
  Since~\(\varrho(A D A)\Hils=\Hils\), we get
  \(\varrho(D)\Hils\neq0\).  Let \(\xi \in \varrho(D) \Hils\) be a
  unit vector.  Then \(f\colon A\to\C\),
  \(a \mapsto \braket{\xi}{\varrho(a)\xi}\), is a pure state that
  restricts to a state on \(D\) and whose GNS representation is
  equivalent to~\(\varrho\).  Hence
  Proposition~\ref{pro:excising_properties} gives a net
  \((a_n)_{n\in N}\) in~\(D\) that strongly excises~\(f\).  Then
  Lemma~\ref{lem:excising_norm_estimate} shows that
  \(\lim {}\norm{a_n x a_n} = 0\).  Thus~\(x\) satisfies Kishimoto's
  condition.
\end{proof}

\section{Aperiodicity and the almost extension property}
\label{sec:aperiodicity_almost extension}

In this section we will relate aperiodicity and topological
freeness to the almost extension property introduced
in~\cite{Nagy-Reznikoff:Pseudo-diagonals}.  The latter is a
weakening of the extension property introduced
in~\cite{Anderson:Extensions_states}, which we also discuss.

\begin{definition}[\cites{Anderson:Extensions_states,
    Nagy-Reznikoff:Pseudo-diagonals}]
  \label{def:almost_extension_property}
  Let \(A\subseteq B\) be a \(\Cst\)\nb-inclusion.  Let
  \(P_1(A \uparrow B)\) be the set of all pure states on~\(A\) that
  extend uniquely to a state on~\(B\).  The \(\Cst\)\nb-inclusion
  has the \emph{almost extension property} if \(P_1(A \uparrow B)\)
  is weak-\(\star\)-dense in the set \(P(A)\) of all pure states
  on~\(A\).  It has the \emph{extension property} if
  \(P_1(A \uparrow B)=P(A)\).
\end{definition}

Whether a pure state extends uniquely depends only on its GNS
representation:

\begin{lemma}
  \label{lem:unique_extension_GNS}
  Let \(f_1\) and~\(f_2\) be two pure states on a
  \(\Cst\)\nb-algebra~\(A\).  If their GNS representations are
  unitarily equivalent and \(f_1 \in P_1(A\uparrow B)\), then
  \(f_2 \in P_1(A\uparrow B)\).
\end{lemma}

\begin{proof}
  By \cite{Pedersen:Cstar_automorphisms}*{Proposition~3.13.4},
  \(f_1\) and~\(f_2\) have equivalent GNS representations if and
  only if there is a unitary \(u\in A+\C\cdot1\) with
  \(f_2(a) = f_1(u a u^*)\) for all \(a\in A\).  Clearly, the subset
  \(P_1(A\uparrow B)\) is invariant under conjugation by unitaries
  in \(A+\C\cdot1\).
\end{proof}

The above lemma allows to replace the weak-\(\star\)-density in \(P(A)\)
in Definition~\ref{def:almost_extension_property} by a number of
other conditions:

\begin{proposition}
  \label{prop:weak_star_density}
  The following are equivalent:
  \begin{enumerate}
  \item \label{enu:weak_star_density1}%
    \(P_1(A\uparrow B)\) is weak-\(\star\)-dense in \(P(A)\);
  \item \label{enu:weak_star_density2}%
    for every non-zero \(a\in A^+\) there is \(f\in P_1(A\uparrow B)\) with
    \(f(a)\neq 0\);
  \item \label{enu:weak_star_density4}%
    for each ideal \(I\idealin A\) with \(I\neq0\), there is
    \(f\in P_1(A\uparrow B)\) with \(f|_I\neq0\);
  \item \label{enu:weak_star_density3}%
    the direct sum of the GNS representations for all
    \(f\in P_1(A\uparrow B)\) is faithful;
  \item \label{enu:weak_star_density5}%
    the image of \(P_1(A\uparrow B)\) is dense in~\(\dual{A}\);
  \end{enumerate}
\end{proposition}

\begin{proof}
  For every non-zero \(a\in A^+\), there is \(f\in P(A)\) with
  \(f(a)\neq 0\).  If \(P_1(A\uparrow B)\) is dense in \(P(A)\),
  then~\(f\) is the weak-\(\star\)-limit of a net \((f_n)\) in
  \(P_1(A\uparrow B)\).  Then \(f_n(a)\neq0\) for some~\(n\).
  So~\ref{enu:weak_star_density1}
  implies~\ref{enu:weak_star_density2}.  The implications
  \ref{enu:weak_star_density2}%
  \(\Longrightarrow\)\ref{enu:weak_star_density4}%
  \(\Longrightarrow\)\ref{enu:weak_star_density3} are obvious.

  For \(f\in P(A)\), let~\(\varrho_f\) be its GNS representation.
  Condition~\ref{enu:weak_star_density5} implicitly uses the map
  \(q\colon P(A) \to \dual{A}\), \(f\mapsto [\varrho_f]\).  Open
  subsets of~\(\dual{A}\) are those of the form~\(\dual{I}\) for
  ideals \(I\subseteq A\).  So the density
  in~\ref{enu:weak_star_density5} means that for each
  \(I\in\Ideals(A)\) with \(I\neq0\), there is
  \(f\in P_1(A\uparrow B)\) with \(\varrho_f|_I \neq0\).  This
  follows from~\ref{enu:weak_star_density3}.  Moreover, the
  map~\(q\) is continuous, open and surjective (see, for instance,
  \cite{Pedersen:Cstar_automorphisms}*{Theorem~4.3.3}).  By
  Lemma~\ref{lem:unique_extension_GNS}, \(P_1(A\uparrow B)\) is the
  preimage of its image in~\(\dual{A}\).  Therefore,
  \ref{enu:weak_star_density1} and~\ref{enu:weak_star_density5} are
  equivalent.
\end{proof}

The following criterion by Anderson for a pure state to extend
uniquely is similar to Kishimoto's condition:

\begin{theorem}[\cite{Anderson:Extensions_states}*{Theorem~3.2}]
  \label{the:unique_extension}
  Let \(A\subseteq B\) be a \(\Cst\)\nb-inclusion.  A pure state
  \(f\colon A\to\C\) extends uniquely to~\(B\) if and only if for
  each \(x\in B\) and \(\varepsilon>0\), there is \(a\in A^+_1\)
  with \(\norm{a x a}_{B/A}<\varepsilon\) and \(f(a)=1\).
\end{theorem}

By the definition of the quotient norm,
\(\norm{a x a}_{B/A}<\varepsilon\) means that there is \(y\in A\)
with \(\norm{a x a-y}_B<\varepsilon\).

Both Kishimoto's condition and Anderson's criterion for the almost
extension property ask that for a given \(x\in B\) and
\(\varepsilon>0\) there should be \(a\in A^+_1\) with
\(\norm{a x a}_{B/A}<\varepsilon\).  In addition, Kishimoto's
condition asks that~\(a\) may be taken from a specific non-zero
hereditary subalgebra, whereas the almost extension property asks
that \(f(a)=1\) for a given state~\(f\), belonging to a
weak-\(\star\)-dense set of states.

The following theorem is the first main result of this section:

\begin{theorem}
  \label{the:almost_extension_aperiodic}
  \(\Cst\)\nb-inclusions with the almost extension property are
  aperiodic.  If~\(B\) is separable, then \(A\subseteq B\) is an
  aperiodic inclusion if and only if \(A\subseteq B\) has the almost
  extension property.
\end{theorem}

\begin{proof}
  Assume first that the inclusion \(A\subseteq B\) has the almost
  extension property.  Let \(x\in B\), \(\varepsilon>0\), and
  \(D\in \Her(A)\).  We are going to check Kishimoto's condition for
  this data.  Let \(A D A\) denote the two-sided ideal generated
  by~\(D\).  Since we assumed the almost extension property, there
  is a pure state in \(P_1(A\uparrow B)\) whose GNS representation
  \(\varrho\colon A \to \Bound(\Hils)\) belongs to~\(\dual{A D A}\).
  Equivalently, \(\varrho|_D\neq0\).  Choose a unit vector
  \(\xi\in \varrho(D)\Hils\) and let~\(f\) be its vector state.  It
  belongs to \(P_1(A\uparrow B)\) by
  Lemma~\ref{lem:unique_extension_GNS}.  The restriction~\(f|_D\) is
  a state whose associated cyclic representation is the restriction
  of~\(\varrho\) to~\(D\) acting on \(\varrho(D)\Hils\) with cyclic
  vector~\(\xi\).  This representation is the image
  of~\(\varrho|_{A D A}\) under the Rieffel correspondence for the
  Morita--Rieffel equivalence between \(A D A\) and~\(D\).  Thus it
  is again irreducible.  The state~\(f|_D\) extends uniquely
  to~\(A\) by
  \cite{Pedersen:Cstar_automorphisms}*{Proposition~3.1.6}.
  Since~\(f\) extends uniquely to~\(B\), the state~\(f|_D\) belongs
  to \(P_1(D\uparrow B)\).  Anderson's criterion in
  Theorem~\ref{the:unique_extension} gives \(a\in D^+_1\) with
  \(\norm{a x a}_{B/D}<\varepsilon\) (and \(f(a)=1\)).  This implies
  Kishimoto's condition for~\(x\).

  Conversely, let \(A\subseteq B\) be an aperiodic inclusion and
  let~\(B\) be separable.  We follow the proof of
  \cite{Olesen-Pedersen:Applications_Connes_3}*{Proposition~6.5} to
  show that the inclusion has the almost extension property.
  Let~\((x_n)_{n\in\N}\) be a dense sequence in the unit ball
  of~\(B\).  Let \(D\in \Her(A)\).  We recursively construct a
  sequence \((e_n)_{n\in \N} \in D^+_1\) such that
  \[
    e_n e_{n+1} = e_{n+1}\quad \text{ and } \quad
    \norm{e_n x_n e_n}_{B/A} \le 2^{-n}
  \]
  for all \(n\in\N\).  To this end, we simultaneously construct
  auxiliary elements \(d_n, y_n \in D^+_1\) with \( e_n y_n = y_n\)
  and \(e_n, y_n\in \Cst(d_n)\).  Pick any \(d_0\in D^+_1\).  The
  functional calculus for~\(d_0\) gives elements
  \(e_0,y_0 \in \Cst(d_0) \subseteq D^+_1\) with \(e_0 y_0 = y_0\)
  as in the proof of
  \cite{Kwasniewski-Meyer:Aperiodicity}*{Lemma~2.9}.  The estimate
  \(\norm{e_0 x_0 e_0}_{B/A} \le 2^{-0}\) is trivial.  Assume
  \(e_n,y_n\) have been constructed as above.  Let
  \(D_{n+1} \defeq \setgiven{z\in D}{z e_n = e_n z = z}\).  This is
  a non-zero hereditary \(\Cst\)\nb-subalgebra because it contains
  \(y_n\neq0\).  Since \(B/A\) is aperiodic, there is
  \(d_{n+1}\in (D_{n+1})^+_1\) with
  \(\norm{d_{n+1} x_{n+1} d_{n+1}}_{B/A} < 2^{-(n+1)}\).  As above,
  we use \cite{Kwasniewski-Meyer:Aperiodicity}*{Lemma~2.9} to choose
  elements
  \(e_{n+1},y_{n+1} \in \Cst(d_{n+1})^+_1 \subseteq (D_{n+1})^+_1\)
  with \(e_{n+1} y_{n+1} = y_{n+1}\).  Given \(\varepsilon>0\), we
  may choose \(e_{n+1} = f(d_{n+1})\cdot d_{n+1}\) with
  \(\norm{f}_\infty < 1 + \varepsilon\).  Thus we may
  choose~\(e_{n+1}\) to also satisfy
  \(\norm{e_{n+1} x_{n+1} e_{n+1}}_{B/A} \le 2^{-(n+1)}\).  This
  completes the recursion step.

  Let
  \(K_n \defeq \setgiven{f\in A^*}{\norm{f}\le 1,\ f\ge 0,\
    f(e_n)=1}\).  Any element of~\(K_n\) is a state of~\(A\), and
  the definition implies that~\(K_n\) is closed in the
  weak-\(\star\)-topology and contained in the unit ball.  Therefore,
  \(K_n\) is compact.  If a convex combination of two states \(f,g\)
  belongs to~\(K_n\), then \(f,g\in K_n\) because \(f(e_n) \le 1\)
  and \(g(e_n) \le 1\) for all states.  Thus~\(K_n\) is a compact
  facet of the set of states.  If \(f\in K_n\), then~\(e_n\) fixes
  the cyclic vector in the GNS representation of~\(f\).  This
  implies \(f(y) = f(e_n y) = f(y e_n)\) for all \(y\in A\).  Then
  \(f(e_{n-1}) = f(e_{n-1} e_n) = f(e_n) = 1\).  This shows that
  \(K_n \subseteq K_{n-1}\).  Then the intersection \(\bigcap K_n\)
  of the decreasing chain of compact facets~\(K_n\) of the state
  space of~\(A\) must be a non-empty facet.  Thus it contains a pure
  state~\(\varphi\).  By construction, \(\varphi(e_n) = 1\) for all
  \(n\in\N\).

  Let \(x\in B\) and \(\varepsilon>0\).  We claim that there is
  \(n\in\N\) with \(\norm{e_n x e_n}_{B/A} < \varepsilon\).
  Rescaling~\(x\), we may assume without loss of generality that
  \(\norm{x}\le1\).  Since~\(\{x_n\}\) is dense in the unit ball
  of~\(B\), there is a subsequence~\((n(k))_{k\in\N}\) with
  \(\lim x_{n(k)} = x\) and \(\lim n(k) = \infty\).  There is
  \(k_0\in\N\) with \(\norm{x-x_{n(k)}}<\varepsilon/2\) for
  \(k\ge k_0\).  There is \(k\ge k_0\) with
  \(2^{-n(k)} < \varepsilon/2\).  Then
  \[
    \norm{e_{n(k)} x e_{n(k)}}
    \le \norm{e_{n(k)} (x-x_{n(k)}) e_{n(k)}}
    + \norm{e_{n(k)} x_{n(k)} e_{n(k)}}
    < \varepsilon/2 + 2^{-n(k)}
    < \varepsilon.
  \]
  By Anderson's criterion in Theorem~\ref{the:unique_extension}, the
  claim that we have just shown implies that~\(\varphi\) extends
  uniquely to a state on~\(B\).  Since \(\norm{\varphi|_{D}}=1\) by
  construction and~\(D\) was an arbitrary hereditary
  \(\Cst\)\nb-subalgebra, this shows that the inclusion
  \(A\subseteq B\) has the almost extension property (see
  Proposition~\ref{prop:weak_star_density}).
\end{proof}

\begin{remark}
 The
  separability assumption in the second part of
  Theorem~\ref{the:almost_extension_aperiodic} is needed, see  Example~\ref{exa:top_free_vs_principal} below.
\end{remark}

Next we study when pure states extend uniquely to crossed products.
We first work in the generality of inverse semigroup actions by
Hilbert bimodules.  Then we specialise to Fell bundles over groups
and étale groupoids.  The following proofs are inspired by the proof
of \cite{Akemann-Weaver:Naimark_problem}*{Theorem~2}.

\begin{proposition}
  \label{pro:trivial_isotropy_to_unique_extension}
  Let~\(\Hilm\) be an action of a unital inverse semigroup~\(S\) on
  a \(\Cst\)\nb-algebra~\(A\) by Hilbert bimodules.  Let
  \(A\subseteq B\) be a \(\Cst\)\nb-inclusion with a surjective
  \Star{}homomorphism \(A\rtimes S \onto B\) which restricts to the
  identity on~\(A\).  Let~\(f\) be a pure state on~\(A\) and let
  \(\varrho\colon A\to\Bound(\Hils)\) be its GNS representation.
  If\/ \([\varrho]\) has trivial isotropy in the dual groupoid of
  the action~\(\Hilm\), then \(f\in P_1(A\uparrow B)\).
\end{proposition}

\begin{proof}
  Let~\(N_{[\varrho]}\) be the directed set of open neighbourhoods
  of~\([\varrho]\).  Each element of~\(N_{[\varrho]}\) has the
  form~\(\dual{J}\) for an ideal~\(J\) in~\(A\) with
  \(\varrho|_J\neq0\); equivalently, \(f|_J\) is a pure state
  on~\(J\).  Proposition~\ref{pro:excising_properties} gives a net
  \((a_n)_{n\in N_J}\) in~\(J\) that strongly excises~\(f\).  We
  combine all these nets, indexing them by the disjoint union
  \(\bigsqcup_{J\in N_{[\varrho]}} N_J\) with a suitable partial
  order.  The result is a net \((a_n)_{n\in N}\) in~\(A\) that
  strongly excises~\(f\) and such that for each
  \(J\in N_{[\varrho]}\) there is \(n_0\in N\) with \(a_n \in J\)
  for all \(n \ge n_0\).  We claim that
  \(\lim {}\norm{a_n x a_n}_{B/A} = 0\) for all \(x\in B\).

  The subset of elements \(x\in B\) with
  \(\lim {}\norm{a_n x a_n}_{B/A} = 0\) is a norm closed vector
  subspace.  Since the images of~\(\Hilm_t\) in~\(B\) for \(t\in S\)
  are linearly dense, it suffices to check the claim for \(t\in S\)
  and \(x\in \Hilm_t\).  If~\(t\) is such that
  \(\Hilm_t \otimes_A \varrho\) is not unitarily equivalent
  to~\(\varrho\), then the claim follows from
  Lemma~\ref{lem:excising_norm_estimate}.  So assume
  \(\Hilm_t \otimes_A \varrho \cong \varrho\).  Since we
  assumed~\([\varrho]\) to have trivial isotropy in the dual
  groupoid, it follows that \([\varrho] \in \dual{I_{1,t}}\).  (For
  a group action, this only happens for \(t=1\).)  Then there is
  \(n_0\in N\) with \(a_n \in I_{1,t}\) for \(n\ge n_0\).  Thus
  \(a_n x a_n \in \Hilm_t\cdot I_{1,t}\).  This is identified
  in~\(A\rtimes S\) and hence in~\(B\) with \(I_{1,t} \subseteq A\).
  So \(\norm{a_n x a_n}_{B/A} = 0\) for all \(n \ge n_0\).  This
  proves the claim.  Then~\(f\) extends uniquely to a state
  on~\(A\rtimes S\) by Anderson's criterion in
  Theorem~\ref{the:unique_extension}.
\end{proof}

\begin{proposition}
  \label{pro:unique_extension_to_trivial_isotropy}
  Let~\(\Hilm\) be an action of a unital inverse semigroup~\(S\) on
  a \(\Cst\)\nb-algebra~\(A\) by Hilbert bimodules.  Let
  \(A\subseteq B\) be a \(\Cst\)\nb-inclusion with a surjective
  \Star{}homomorphism \(B \onto A\rtimes_\red S\) that restricts to
  the canonical inclusion \(A\hookrightarrow A\rtimes_\red S\).
  Let~\(f\) be a pure state on~\(A\) and let
  \(\varrho\colon A\to\Bound(\Hils)\) be its GNS representation.  If\/
  \([\varrho]\in\dual{A}\) has non-trivial isotropy in the dual
  groupoid of the action~\(\Hilm\), then
  \(f\notin P_1(A\uparrow B)\).
\end{proposition}

\begin{proof}
  By assumption, there is \(t\in S\) with
  \(\dual{\Hilm_t}[\varrho] = [\varrho]\) and
  \([\varrho] \notin \dual{I_{1,t}}\) for the ideal~\(I_{1,t}\)
  defined in~\eqref{eq:Itu}; if~\(S\) is a group, this simply means
  \(t\neq 1\).  Then \(\Hilm_t \otimes_\varrho \Hils\neq0\), so that
  there is \(x\in \Hilm_t\) with \(x\otimes \xi\neq0\).  The
  left multiplication representation of~\(A\) on
  \(\Hilm_t \otimes_\varrho \Hils\) is unitarily equivalent
  to~\(\varrho\).  Since~\(\varrho\) is irreducible, Kadison's
  Transitivity Theorem allows us to choose an element \(a\in A\) so
  that the unitary intertwiner between these representations maps
  \(a\cdot (x\otimes \xi)\) to the canonical cyclic vector~\(\xi\).
  We could have picked \(a\cdot x\) instead of~\(x\) from the
  beginning, and we assume this to simplify notation.  Then the
  unitary intertwiner maps \(x\otimes \xi\) to~\(\xi\).  So both
  vectors define the same vector state on~\(A\).  That is,
  \begin{equation}
    \label{eq:tensor_with_f_equivalent}
    f(a)
    = \braket[\big]{\xi}{\varrho(a)\xi}
    = \braket[\big]{x\otimes \xi}{a\cdot x\otimes\xi}
    = \braket[\big]{\xi}{\varrho(\braket{x}{a\cdot x})\xi}
    = f\bigl(\braket{x}{a\cdot x}\bigr)
  \end{equation}
  for all \(a\in A\).  Let \(\tilde{x} \in B\) be a pre-image
  for~\(x\) under the surjective map \(B \to A\rtimes_\red S\).  We
  claim that \(\norm{a \tilde{x} a}_{B/A} \ge 1\) holds for all
  \(a\in A^+_1\) with \(f(a) = 1\).

  The GNS representation~\(\varrho\) of~\(f\) induces a
  representation of~\(A\rtimes_\red S\) and thus a
  representation~\(\omega\) of~\(B\).  We are going to verify
  \(\norm{\omega(a x a -y)} \ge 1\) for all \(y\in A\) and
  \(a\in A^+_1\) with \(f(a)=1\); this will finish the proof of the
  theorem.  To build the representation~\(\omega\), we first extend
  the \(S\)\nb-action on~\(A\) to the bidual~\(A''\).  (This step is
  not needed in for group actions because then~\(E\) takes values in~\(A\).)
  We use the canonical conditional expectation
  \(E\colon A''\rtimes_\red S \to A''\) to form a Hilbert
  \(A''\)\nb-module \(\ell^2(S,A'')\).  Then~\(\omega\) is the left
  multiplication action on the tensor product
  \(\Hils[K] \defeq \ell^2(S,A'') \otimes_{A''} (\Hils,\varrho)\).
  Since~\(A''\) is unital, the Hilbert space~\(\Hils[K]\) contains a
  copy of~\(\Hils\) of the form \(1\otimes \Hils\).  Let
  \(a,y\in A\).  Then \(a x a - y \in A\rtimes_\alg S\) maps the
  unit vector \(1\otimes \xi\) to the vector
  \(a x a \otimes \xi - y \otimes \xi\) in~\(\Hils[K]\).  We claim
  that the summands \(a x a \otimes \xi\) and \(y \otimes \xi\) are
  orthogonal.  Indeed, their inner product is defined to be
  \[
    \braket[\big]{\xi}{\varrho''\circ E((a x a)^*\cdot y)\xi}
    = f\circ E((a x a)^*\cdot y).
  \]
  This vanishes because the expectation~\(E\) multiplies with the
  support projection~\([I_{1,t}]\) of the ideal~\(I_{1,t}\), which
  is killed by~\(f\) because \([\varrho] \notin \dual{I_{1,t}}\).
  So
  \begin{align*}
    \norm{\omega(a x a - y)}
    &\ge \norm{(a x a - y)(1\otimes \xi)}_{\ell^2(S,A'')
      \otimes_{\varrho''} \Hils}
    \\
    &\ge \norm{a x a \otimes \xi}_{\ell^2(S,A'') \otimes_{\varrho''} \Hils}
      = f\bigl(\braket{a x a}{a x a}\bigr)^{1/2}.
  \end{align*}
  Recall that we assume \(f(a)=1\).  Then
  \(\norm{\varrho(a) \xi} = 1\) and
  \(\braket{\varrho(a) \xi}{\xi}=1\), and this implies
  \(\varrho(a) \xi = \xi\).  So \(f(y a) = f(a) = f(a y)\) for all
  \(y\in A\).  Using this and~\eqref{eq:tensor_with_f_equivalent},
  we compute
  \[
    f\bigl(\braket{a x a}{a x a}\bigr)
    = f\bigl(a \braket{a x}{a x} a\bigr)
    = f\bigl(\braket{a x}{a x}\bigr)
    = f\bigl(\braket{x}{a^2 x}\bigr)
    = f(a^2) = 1.
  \]
  This finishes the proof of the claim.  Then~\(f\) has more than
  one extension to a state on~\(B\) by Anderson's criterion in
  Theorem~\ref{the:unique_extension}.
\end{proof}

\begin{definition}
  A groupoid is called \emph{principal} if all points in \(G^0\)
  have trivial isotropy.  A topological groupoid~\(G\) is called
  \emph{topologically principal} if the set of \(x\in G^0\) with
  trivial isotropy group is dense in~\(G^0\).
\end{definition}

\begin{remark}
  \label{rem:about_Baire_and_Hausdorff}
  If an étale groupoid~\(G\) is topologically principal, then it is
  topologically free.  The converse holds when \(G\) has a countable
  cover by bisections and the unit space \(X\) contains a dense
  Hausdorff Baire space (see
  \cite{Kwasniewski-Meyer:Essential}*{Corollary~2.26}).  The case
  when~\(X\) is not Hausdorff is more subtle (see the comment before
  \cite{Kwasniewski-Meyer:Essential}*{Proposition~2.24}).  In
  particular, the proof of
  \cite{Renault:Cartan.Subalgebras}*{Proposition~3.6.(ii)} has a
  gap.
\end{remark}

\begin{theorem}
  \label{the:almost_extension_top_principal}
  Let~\(S\) be a unital inverse semigroup that acts on a \(\Cst\)\nb-algebra~\(A\) by Hilbert
  bimodules.  Let~\(B\) be a \(\Cst\)\nb-algebra with surjective
  \Star{}homomorphisms \(A\rtimes S \onto B \onto A\rtimes_\red S\)
  that compose to the quotient map
  \(A\rtimes S\onto A\rtimes_\red S\).
  \begin{enumerate}
  \item \label{item:almost_extension_top_principal1}%
    A pure state \(f\in P(A)\) belongs to \(P_1(A \uparrow B)\) if
    and only if the GNS representation of~\(f\) has trivial isotropy
    in the dual groupoid \(\dual{A}\rtimes S\).
  \item \label{item:almost_extension_top_principal2}%
    The inclusion \(A \hookrightarrow B\) has the almost extension
    property if and only if the dual groupoid \(\dual{A}\rtimes S\)
    is topologically principal.
  \item \label{item:almost_extension_top_principal3}%
    The inclusion \(A \hookrightarrow B\) has the extension property
    if and only if the dual groupoid is principal.
  \end{enumerate}
\end{theorem}

\begin{proof}
  Propositions \ref{pro:trivial_isotropy_to_unique_extension}
  and~\ref{pro:unique_extension_to_trivial_isotropy} combined
  give~\ref{item:almost_extension_top_principal1}.
  Statement~\ref{item:almost_extension_top_principal1}
  implies~\ref{item:almost_extension_top_principal3}.  It also
  implies~\ref{item:almost_extension_top_principal2} because
  \(P_1(A\uparrow B)\) is dense in~\(P(A)\) if and only if its image
  is dense in~\(\dual{A}\) by
  Proposition~\ref{prop:weak_star_density}.
\end{proof}

Theorem~\ref{the:almost_extension_top_principal} generalises a
result of Zarikian for group actions by automorphisms on unital
\(\Cst\)\nb-algebras (see \cite{Zarikian:Pure_extension}*{Theorem
  2.4}).

\begin{corollary}
  \label{cor:extension_prop_Fell_groupoid}
  Let~\(\A\) be a Fell bundle over an étale groupoid~\(G\) with a
  locally compact Hausdorff unit space~\(X\).  Put
  \(A\defeq \Cont_0(X,\A|_X)\) and let~\(B\) be a
  \(\Cst\)\nb-algebra with surjective maps
  \(\Cst(G, \A) \onto B \onto \Cst_\red(G, \A)\) that compose to the
  quotient map \(\Cst(G, \A) \onto \Cst_\red(G, \A)\).
  \begin{enumerate}
  \item A pure state \(f\in P(A)\) belongs to \(P_1(A \uparrow B)\)
    if and only if the GNS representation of~\(f\) has trivial
    isotropy in the dual groupoid \(\dual{A}\rtimes G\).
  \item The inclusion \(A \hookrightarrow B\) has the extension
    property if and only if the dual groupoid~\(\dual{A}\rtimes G\)
    is principal.
  \item The inclusion \(A \hookrightarrow B\) has the almost
    extension property if and only if the dual
    groupoid~\(\dual{A}\rtimes G\) is topologically principal.
  \end{enumerate}
\end{corollary}

\begin{proof}
  A Fell bundle over an étale, locally compact groupoid gives rise
  to an action of an inverse semigroup such that the full, reduced,
  and essential crossed products and the dual groupoids are the same
  (see \cite{Kwasniewski-Meyer:Essential}*{Section~7}).  The inverse
  semigroup is generated by bisections of~\(G\).  So we may choose
  it to be countable if~\(G\) is covered by countably many
  bisections.  Now all claims follow from
  Theorem~\ref{the:almost_extension_top_principal}.
\end{proof}

\begin{remark}
  Assume the dual groupoid \(\dual{A}\rtimes G\) of a Fell bundle
  over an étale, locally compact groupoid~\(G\) to be principal.
  Then \cite{Kwasniewski-Meyer:Essential}*{Lemma~7.15 and
    Proposition~7.18} imply that
  \(\Cst_\red(G,\A) = \Cst_\ess(G,\A)\).
\end{remark}

%\begin{remark}
  In Proposition~\ref{pro:trivial_isotropy_to_unique_extension},
  \(B\) may be \(A\rtimes S\), \(A\rtimes_\red S\), or
  \(A\rtimes_\ess S\).  In fact, it may be any \emph{\(S\)\nb-graded
  \(\Cst\)\nb-algebra}~\(B\) (\cite{Kwasniewski-Meyer:Stone_duality}*{Definition~6.15}).
	The \(S\)\nb-grading is a family of
  closed subspaces \((B_t)_{t\in S}\) with \(B_t^* = B_{t^*}\) and
  \(B_t \cdot B_u = B_{t u}\) for all \(t,u\in S\) and
  \(\sum B_t = B\).
	Then the Banach spaces~\(B_t\) with the
  multiplication and involution from~\(B\) define an action of~\(S\)
  on~\(A\) by Hilbert bimodules.  The inclusion maps
  \(B_t \hookrightarrow B\) induce a canonical surjective
  \Star{}homomorphism from the crossed product for this action
  to~\(B\).
%\end{remark}

In contrast,
Proposition~\ref{pro:unique_extension_to_trivial_isotropy} may fail
for essential crossed products.  That is, a pure state~\(f\)
on~\(A\) may extend uniquely to the essential crossed product
\(A\rtimes_\ess S\) without extending uniquely to
\(A\rtimes_\red S\).  We do not know how to characterise which pure
states extend uniquely to \(A\rtimes_\ess S\).  Even the equivalence in
Theorem~\ref{the:almost_extension_top_principal}\ref{item:almost_extension_top_principal2} may fail, as the
following counterexample shows:

\begin{example}
  \label{exa:extension_property_not_top_principal}
  There is a non-Hausdorff groupoid~\(G\) such that
  \(\Cst_\ess(G) = \Cont_0(G^0)\) although~\(G\) is not topologically
  principal.  The construction starts with the uncountable group
  \(\Gamma \defeq \bigoplus_{t\in [0,1]} \Z/2\).  Equip the trivial
  group bundle \([0,1]\times \Gamma\) over~\([0,1]\) with the
  equivalence relation defined by
  \[
    \biggl(t,\sum_{s\in [0,1]} a_s\biggr) \sim
    \biggl(x,\sum_{s\in [0,1]} b_s\biggr) \quad\iff\quad
    x=t \text{ and } a_t=b_t.
  \]
  Let~\(G\) be the quotient of \([0,1]\times\Gamma\) by~\(\sim\),
  equipped with the quotient topology.  	This is a group bundle
  over~\([0,1]\) with fibres~\(\Z/2\) at all points.  Let \(q\colon [0,1]\times\Gamma \onto G\) be the quotient map.
	The sets
  \(q([0,1]\times \{\gamma\})\) for \(\gamma\in\Gamma\) are open
  bisections of~\(G\) that cover \(G\).  So~\(G\) is an étale, non-Hausdorff
  groupoid.  It is not topologically principal because each isotropy
  group is isomorphic to~\(\Z/2\).  The unit bisection~\(G^0\) is
  dense in~\(G\) because it  intersects \(q([0,1]\times \{\gamma\})\)
  for each \(\gamma\in\Gamma\).  It follows that the essentially defined
  conditional expectation \(\Cst(G) \to \Locmult(\Cont[0,1])\) is a
  \Star{}homomorphism to \(\Cont[0,1]\).  Thus the inclusion map
  \(\Cont[0,1] \hookrightarrow \Cst_\ess(G)\) becomes an
  isomorphism.  As a result, any (pure) state on \(\Cont[0,1]\)
  extends uniquely to \(\Cst_\ess(G)\), although~\(G\) is not
  topologically principal.
\end{example}

It is crucial that the groupoid in
Example~\ref{exa:extension_property_not_top_principal} is not
covered by a countable family of bisections.  Indeed,
Theorem~\ref{the:cycle_of_equivalences} below implies the
equivalence in
Theorem~\ref{the:almost_extension_top_principal}.\ref{item:almost_extension_top_principal2}
for all exotic crossed products provided~\(S\) is countable
and~\(A\) contains an essential ideal that is separable or of
Type~I.  In fact, under these assumptions the almost extension
property is equivalent to a number of conditions including
topological freeness and aperiodicity.

\begin{example}
  \label{exa:top_free_vs_principal}
  Let~\(G\) be the group of affine isometries of~\(\R\).  It is
  generated by translations and one reflection and therefore
  isomorphic to \(\R\rtimes \Z/2\).  Give~\(G\) the discrete
  topology.  The transformation groupoid for the action of~\(G\)
  on~\(\R\) is topologically free because each element of~\(G\setminus\{1\}\)
  fixes at most one point in~\(\R\).  It is not topologically
  principal because each point in~\(\R\) is fixed by some element
  of~\(G\setminus\{1\}\).  Our results show that the induced action of~\(G\) on
  \(\Cont_0(\R)\) is topologically free and aperiodic.  There is,
  however, no pure state on \(\Cont_0(\R)\) that extends uniquely to
  \(\Cont_0(\R)\rtimes G\).
\end{example}

\section{Detection of ideals in intermediate algebras}
\label{sec:detection_in_intermediate}

The work of Olesen--Pedersen
\cites{Olesen-Pedersen:Applications_Connes,
  Olesen-Pedersen:Applications_Connes_2,
  Olesen-Pedersen:Applications_Connes_3} shows that for actions of
the group~\(\Z\) or~\(\Z/p\) for a square-free number~\(p\) on a
separable \(\Cst\)\nb-algebra~\(A\), aperiodicity and topological
freeness are not only sufficient, but also necessary for~\(A\) to
detect ideals in the crossed product (see
also~\cite{Kwasniewski-Meyer:Aperiodicity}).  In this section, we
use this to prove that a stronger condition is always sufficient for an action
to be topologically
free.  Namely, we require~\(A\) to detect ideals in all intermediate
\(\Cst\)\nb-algebras between~\(A\) and the essential crossed
product.  It is natural to strengthen ideal detection in this way
because the uniqueness of pseudo-expectations and aperiodicity are
hereditary for such intermediate inclusions.

\begin{proposition}
  \label{pro:intermediate_detection_to_top-free_isg}
  Let~\(\Hilm\) be an action of a unital inverse semigroup ~\(S\) on
  a \(\Cst\)\nb-algebra~\(A\).  Assume that \(A\) contains an
  essential ideal that is separable or of Type~I.  If~\(A\) detects
  ideals in~\(C\) for any
  \(A\subseteq C \subseteq A\rtimes_\ess S\), then the
  action~\(\Hilm\) is topologically free.  In fact, it suffices to
  assume that~\(A\) detects ideals in \(A\rtimes_\ess T\) for any
  inverse subsemigroup \(T\subseteq S\) that is generated by the
  idempotents in~\(S\) together with a single \(t\in S\).
\end{proposition}

Before we prove this, we explain why \(A\rtimes_\ess T\) is
contained in \(A\rtimes_\ess S\):

\begin{lemma}
  \label{lem:embed_crossed_for_isg}
  Let \(T\subseteq S\) be an inverse subsemigroup.  If~\(T\)
  contains all idempotents in~\(S\), then
  \(A\rtimes_\ess T \subseteq A\rtimes_\ess S\).
\end{lemma}

\begin{proof}
  The inclusion \(T\subseteq S\) induces a canonical
  \Star{}homomorphism \(j\colon A\rtimes T \to A\rtimes S\).  The
  issue is to prove that it descends to an injective
  \Star{}homomorphism between the essential crossed products.
  If~\(T\) contains all idempotents in~\(S\), then it follows that
  the ideals~\(I_{1,t}\) for \(t\in T\) defined in~\eqref{eq:Itu}
  are the same when computed in \(S\) or~\(T\).  The generalised
  expectations \(E^T\colon A\rtimes T \to \Locmult(A)\) and
  \(E^S\colon A\rtimes S \to \Locmult(A)\) are defined in
  \cite{Kwasniewski-Meyer:Essential}*{Proposition~4.3}.  Since the
  ideals~\(I_{1,t}\) are the same in both cases, it follows easily
  that \(E^S\circ j = E^T\).  The generalised expectations on
  \(A\rtimes_\ess S\) and \(A\rtimes_\ess T\) induced by \(E^S\)
  and~\(E^T\) are faithful by
  \cite{Kwasniewski-Meyer:Essential}*{Theorem~4.11}.  That is,
  \(j(x)\) becomes~\(0\) in \(A\rtimes_\ess S\) if and only if
  \(E^S(j(x^* x))=0\), if and only if \(E^T(x^* x)=0\), if and only
  if~\(x\) becomes~\(0\) in \(A\rtimes_\ess T\).  This means
  that~\(j\) descends to an injective map
  \(A\rtimes_\ess T \hookrightarrow A\rtimes_\ess S\).
\end{proof}

\begin{remark}
  If an intermediate \(\Cst\)\nb-algebra
  \(A\subseteq C \subseteq A\rtimes_\ess S\) is equal to
  \(A\rtimes_\ess T\) for some~\(T\) as in
  Lemma~\ref{lem:embed_crossed_for_isg}, then it is \(S\)\nb-graded,
  that is, \(C\) is the closed linear span of \(C\cap \Hilm_t\) for
  \(t\in S\).
\end{remark}

\begin{proof}[Proof of
  Proposition~\textup{\ref{pro:intermediate_detection_to_top-free_isg}}]
  We proceed by contradiction and assume that the action~\(\Hilm\)
  is not topologically free.  Then, by Lemma \ref{lem:topological_free_non_trivial},
	there are \(t\in S\) and a
  non-zero ideal \(J\subseteq I_{1,t}^\bot\) such that the Hilbert
  \(A\)\nb-bimodule \(\Hilm[F]\defeq \Hilm_t\cdot J\) is
  topologically trivial.  That is,
  \(\Hilm[F]\otimes_A \varrho \cong \varrho\) for all irreducible
  representations~\(\varrho\) of~\(J\).  The key step in the proof
  is based on \cite{Kwasniewski-Meyer:Aperiodicity}*{Theorem~9.12}
  which, in turn, is based on the work of Olesen--Pedersen in
  \cites{Olesen-Pedersen:Applications_Connes,
    Olesen-Pedersen:Applications_Connes_2,
    Olesen-Pedersen:Applications_Connes_3}.  That theorem is about a
  section \(\Cst\)\nb-algebra~\(C\) for a Fell bundle over \(\Z\)
  or~\(\Z/p\) for a square-free number~\(p\) whose unit fibre~\(A\)
  contains an essential ideal that is separable or of Type~I.  Then
  the inclusion \(A\subseteq C\) is topologically free if and only
  if~\(A\) detects ideals in~\(C\).  We are going to build a Fell
  bundle over \(\Z\) or~\(\Z/p\) from~\(\Hilm[F]\) in such a way
  that its section \(\Cst\)\nb-algebra~\(C\) is an intermediate
  \(\Cst\)\nb-algebra between \(A\) and \(A\rtimes_\ess S\).

  Let \(\Hilm[F]_0 \defeq J\),
  \(\Hilm[F]_k \defeq \Hilm[F]^{\otimes_A k}\) and
  \(\Hilm[F]_{-k} \defeq (\Hilm[F]^*)^{\otimes_A k}\) for \(k>0\).
  These are slices for the inclusion \(A\subseteq A\rtimes_\ess S\),
  and they form a Fell bundle over~\(\Z\) with the multiplication
  and involution in~\(A\rtimes_\ess S\).  The maps
  \(\Hilm[F]_k \hookrightarrow A\rtimes_\ess S\) form a Fell bundle
  representation.  They induce a \Star{}homomorphism
  \(\varphi_0\colon \Cst(\Z, (\Hilm[F]_k)) \to A\rtimes_\ess S\).
  Assume first that~\(\varphi_0\) is injective.  By assumption, the
  Hilbert bimodule~\(\Hilm[F]_1\) is topologically trivial.  Hence
  the inclusion \(J\hookrightarrow \Cst(\Z, (\Hilm[F]_k))\) cannot
  detect ideals by
  \cite{Kwasniewski-Meyer:Aperiodicity}*{Theorem~9.12}.  Next, we
  replace \(\Cst(\Z, (\Hilm[F]_k))\) by an intermediate
  \(\Cst\)\nb-algebra of the form \(A\rtimes_\ess T\).  Namely, we
  let~\(T\) be the inverse subsemigroup generated by~\(t\) and the
  idempotent elements of~\(S\).  So all elements of~\(T\) are of the
  form \(t^k\cdot e\) for some \(k\in\Z\) and an idempotent
  element~\(e\) of~\(S\).  Now \(J\subseteq A\) is an invariant
  ideal for the action of~\(T\) on~\(A\) and
  \(J\cdot (A\rtimes_\ess T) = (A\rtimes_\ess T) \cdot J = \Cst(\Z,
  (\Hilm[F]_k))\) (see
  \cite{Kwasniewski-Meyer:Stone_duality}*{Proposition~6.19}).
  Therefore, since \(J\hookrightarrow \Cst(\Z, (\Hilm[F]_k))\) does
  not detect ideals, neither does \(A\subseteq A\rtimes_\ess T\).
  This finishes the proof in the case where~\(\varphi_0\) is
  injective.

  It remains to study the case when~\(\varphi_0\) is not injective.
  Let \(E\colon A\rtimes_\ess S \to \Locmult(A)\) be the canonical
  essentially defined expectation.  By definition, if \(u\in S\),
  then~\(E\) is the identity map on \(\Hilm_u \cap I_{u,1}\) and
  vanishes on \(\Hilm_{u} \cap I_{u,1}^\bot\).  By construction,
  \(\Hilm[F]_k \subseteq \Hilm_{t^k}\) for \(k\ge0\) and
  \(\Hilm[F]_1\subseteq  \Hilm_{t} I_{1,t}^\bot\).  We claim that
  there must be \(k\ge2\) for which
  \(\Hilm[F]_k \not\subseteq \Hilm_{t^k} \cdot I_{t^k,1}^\bot\).
  Otherwise,
  \(\Hilm[F]_k \subseteq \Hilm_{t^k} \cdot I_{t^k,1}^\bot\) holds
  for all \(k\ge1\).  Then
  \(E|_{\Hilm[F]_k} = 0\) for all \(k>0\).  Then
  \(E \circ \varphi_0\) is equal to the canonical conditional
  expectation \(\Cst(\Z,(\Hilm[F]_k)) \to A\).  The latter is
  faithful because~\(\Z\) is amenable, and so~\(\varphi_0\) is
  injective.  %But we assumed~\(\varphi_0\) is not injective.
  Hence there are \(k\ge1\) for which~\(\Hilm[F]_k\) is not
  contained in \(\Hilm_{t^k} \cdot I_{t^k,1}^\bot\).  We pick the
  minimal such~\(k\).  So \(E|_{\Hilm[F]_n}=0\) for
  \(n=1,\dotsc,k-1\) and
  \(\Hilm[F]_k \cap \Hilm_{t^k} \cdot I_{t^k,1} \neq 0\).  This
  intersection is equal to \(\Hilm_{t^k} \cdot K\) for some non-zero
  ideal \(K\subseteq I_{t^k,1} \cap J\).  We replace~\(J\) by~\(K\).
  This improves matters in such a way that
  \(\Hilm[F]_k \subseteq \Hilm_{t^k} \cdot I_{t^k,1}\).  This is
  contained in \(A \subseteq A\rtimes_\ess S\), and it is, in fact,
  equal to the chosen ideal~\(K\).  It follows that
  \(\Hilm[F]_n \cdot \Hilm[F]_k = \Hilm[F]_n\) for all \(n\ge0\).
  Then \((\Hilm[F]_n)_{n=1,\dotsc,k}\) is a Fell bundle
  over~\(\Z/k\) and the inclusions
  \(\Hilm[F]_n \hookrightarrow A\rtimes_\ess S\) form a Fell bundle
  representation.  This induces an injective \Star{}homomorphism
  \(\Cst(\Z/k,(\Hilm[F]_n)) \hookrightarrow A\rtimes_\ess S\).
  If~\(k\) is square-free, then
  \cite{Kwasniewski-Meyer:Aperiodicity}*{Theorem~9.12} shows
  that~\(K\) does not detect ideals in \(\Cst(\Z/k,(\Hilm[F]_n))\).
  In general, we write \(k=p\cdot k_1\) with a prime number~\(p\).
  An argument as above shows that~\(K\) does not detect ideals in
  \(\Cst(\Z/p,(\Hilm[F]_{k_1\cdot n})_{n=1,\dotsc,p})\).  Let~\(T\)
  be the inverse subsemigroup generated by~\(t^{k_1}\) and all
  idempotent elements of~\(S\).  The same argument as in the case
  where~\(\varphi_0\) is injective shows that~\(A\) does not detect
  ideals in~\(A\rtimes_\ess T\).
\end{proof}

\begin{proposition}
  \label{pro:intermediate_detection_to_top-free_groupoid}
 Let~\(\A\) be
  a Fell bundle over an étale groupoid~\(G\)  with a locally compact Hausdorff unit space~\(X\).  Assume
  that \(A\defeq \Cont_0(X,\A|_X)\) contains an essential ideal that
  is separable or of Type~I.  Assume that~\(A\) detects ideals in
  \(\Cst_\ess(H,\A|_H)\) for any open subgroupoid
  \(X \subseteq H\subseteq G\) that is generated by a single open
  bisection \(U\subseteq G\).  Then the dual groupoid of the Fell
  bundle~\(\A\) is topologically free.
\end{proposition}

\begin{proof}
  We may rewrite the essential section \(\Cst\)\nb-algebra
  \(B\defeq \Cst_\ess(G,\A)\) as the essential crossed product
  \(A\rtimes_\ess S\) for an inverse semigroup action on~\(A\) (see
  Definition~\ref{def:groupoid_action_aperiodic}).  The dual groupoids
  for both actions are the same, so that topological freeness is
  preserved.
  Proposition~\ref{pro:intermediate_detection_to_top-free_isg} shows
  that the dual groupoid \(A\rtimes G\) is topologically free
  if~\(A\) detects ideals in all intermediate
  \(\Cst\)\nb-algebras \(A\subseteq C \subseteq B\).  It remains to
  show that the special intermediate algebras of the form
  \(\Cst_\ess(H,\A|_H)\) for open subgroupoids
  \(X \subseteq H\subseteq G\) suffice to carry through the proof of
  Proposition~\ref{pro:intermediate_detection_to_top-free_isg}.

  For this, we first describe the inverse semigroup~\(S\).  Each
  bisection \(U\subseteq G\) defines a subspace
  \(\A_U \subseteq \Cst_\ess(G,\A)\) of sections supported in~\(U\).
  This subspace belongs to the inverse semigroup
  \(\mathcal{S}(A,B)\) in
  \cite{Kwasniewski-Meyer:Essential}*{Proposition~2.12}.  We
  let~\(S\) be the inverse subsemigroup of \(\mathcal{S}(A,B)\)
  generated by \(\A_U\cdot J\) for all ideals~\(J\)
  in~\(A\).  By definition, \(B\) is graded by the inverse
  semigroup~\(S\).  This gives a surjective \Star{}homomorphism
  \(\varphi\colon A\rtimes S \to B\).  Composing the canonical
  generalised expectation \(B = \Cst_\ess(G,\A) \to \Locmult(A)\)
  with~\(\varphi\) gives the canonical generalised expectation
  \(A\rtimes S \to \Locmult(A)\).  Therefore, \(\varphi\) induces an
  isomorphism \(A\rtimes_\ess S \cong B\).  Now assume that the dual
  groupoid \(\dual{A}\rtimes G\) is not topologically free.  Then
  there is a non-zero open bisection
  \(U\subseteq \dual{A}\rtimes G \setminus \dual{A}\) with
  \(r|_U= s|_U\).  We may assume that the bisection~\(U\) ``lives''
  on a bisection \(V\subseteq G\) on~\(\dual{A}\).  More precisely,
  the partial homeomorphism corresponding to~\(U\) is the
  homeomorphism dual to the Hilbert \(A\)\nb-bimodule
  \(\A_V \cdot J\) for an ideal \(J\subseteq s(\A_V)\).  The union
  \(H\defeq X \cup \bigcup_{k\in\Z} V^k\) is an open subgroupoid
  of~\(G\) containing the units.  As in the proof of
  Proposition~\ref{pro:intermediate_detection_to_top-free_isg}, we
  may arrange that either \(V^p \subseteq X\) for some square-free
  number~\(p\) or \(V^k \cap X = \emptyset\) for all
  \(k\in\N_{\ge1}\).  Then \(\A_k\defeq \A_{V^k}\) for
  \(k=1,\dotsc,p\) or for \(k\in\Z\) defines a Fell bundle over the
  group~\(\Z/p\) or over~\(\Z\).  The section \(\Cst\)\nb-algebra of
  this Fell bundle is isomorphic to \(\Cst_\ess(H,\A|_H)\).  Now we
  can argue as in the proof of
  Proposition~\ref{pro:intermediate_detection_to_top-free_isg}
  that~\(A\) cannot detect ideals in \(\Cst_\ess(H,\A|_H)\) because
  the dual groupoid for the action of~\(H\) on~\(A\) is not
  topologically free.
\end{proof}

\begin{corollary}
  \label{cor:intermediate_detection_to_top-free_group}
  Let~\(\A\) be a Fell bundle over a discrete group~\(G\).  Assume
  that \(A\defeq A_1\) contains an essential ideal that is separable
  or of Type~I.  If~\(A\) detects ideals in \(\Cst_\red(H,\A|_H)\)
  for any cyclic subgroup \(H\subseteq G\), then the dual groupoid
  of the Fell bundle~\(\A\) is topologically free.
\end{corollary}

If~\(A\) is arbitrary, then it is unclear whether detection of
ideals in intermediate \(\Cst\)\nb-algebras implies that the action
is topologically free.  Only a weaker statement follows.  Namely,
the action is purely outer in the following sense:

\begin{definition}
  A Hilbert \(A\)\nb-bimodule~\(\Hilm[H]\) over a
  \(\Cst\)\nb-algebra~\(A\) is \emph{purely
    outer}~\cite{Kwasniewski-Meyer:Aperiodicity} if there is no
  non-zero ideal \(J\in\Ideals(A)\) with \(\Hilm[H]\cdot J \cong J\)
  as a Hilbert \(A\)\nb-bimodule.  An action~\(\Hilm\) of an inverse
  semigroup on a \(\Cst\)\nb-algebra~\(A\) is \emph{purely
    outer}~\cite{Kwasniewski-Meyer:Essential} if the Hilbert
  \(A\)\nb-bimodules \(\Hilm_t\cdot I_{1,t}^\bot\) are purely outer
  for all \(t\in S\).
\end{definition}

\begin{proposition}
  \label{pro:intermediate_detection_to_purely_outer}
  Let~\(\Hilm\) be an action of a unital inverse semigroup~\(S\) on
  a \(\Cst\)\nb-algebra~\(A\).  Assume that~\(A\) detects ideals in
  \(A\rtimes_\ess T\) for any inverse subsemigroup \(T\subseteq S\)
  that is generated by the idempotents in~\(S\) together with a
  single \(t\in S\).  Then the action~\(\Hilm\) is purely outer.
\end{proposition}

\begin{proof}
  Assume that the action is not purely outer.  Then there are
  \(t\in S\) and a non-zero ideal \(J\subseteq I_{1,t}^\bot\) such
  that there is an isomorphism  of Hilbert \(A\)\nb-bimodules
  \[
    \varphi\colon \Hilm_t \cdot J \congto J.
  \]
    As in the proof of
  Proposition~\ref{pro:intermediate_detection_to_top-free_isg}, we
  may shrink~\(J\) and replace~\(t\) by~\(t^{k_1}\) for some
  \(k_1\in\N\) to arrange that either \(J \subseteq I_{1,t^k}^\bot\)
  for all \(k\ge1\) or \(J \subseteq I_{1,t^k}^\bot\) for
  \(k=1,\dotsc,p\) and \(J\subseteq I_{1,t^p}\) for some square-free
  number~\(p\); we put \(p=0\) in the first case.  Let
  \(T\subseteq S\) be generated by the idempotent elements of~\(S\)
  and~\(t\).  Then \(A\rtimes_\ess T\subseteq A\rtimes_\ess S\).
  The ideal~\(J\) is \(T\)\nb-invariant or, equivalently,
  \(J\cdot (A\rtimes_\ess T) = (A\rtimes_\ess T)\cdot J\).  The
  latter is canonically isomorphic to the section
  \(\Cst\)\nb-algebra of a Fell bundle over~\(\Z/p\).

  We may, without loss of generality, assume that \(S=T\), \(A=J\)
  and \(\Hilm_t = \Hilm_t\cdot J\) for all \(t\in S\) to simplify
  notation.  After this change, \(A\rtimes_\ess S\) becomes the
  section \(\Cst\)\nb-algebra of a Fell bundle over~\(\Z/p\).  Since
  we replaced~\(A\) by~\(J\), the isomorphism~\(\varphi\) becomes an
  isomorphism \(\varphi\colon \Hilm_t \congto A\).  This generates
  isomorphisms
  \[
    \Hilm_{t^k} \cong \Hilm_t^{\otimes_A k}
    \cong A^{\otimes_A k} \cong A
  \]
  for all \(k\in\N\).  If \(p=0\), then these maps together form a
  Fell bundle representation of our Fell bundle into~\(A\).  They
  induce a \Star{}homomorphism
  \(A\rtimes_\ess S = A\rtimes S \to A\).  Since it is the identity
  map on~\(A\), its kernel is an ideal~\(J\) in \(A\rtimes_\ess S\)
  with \(J\cap A = 0\).  If \(p\neq0\), then \(\Hilm_{t^p} = A\)
  inside \(A\rtimes_\ess S\).  So~\(\varphi\) above generates a
  Hilbert bimodule isomorphism \(A \congto A\).  This is the same as
  a central unitary multiplier of~\(A\).  Shrinking the ideal~\(J\)
  further, we may arrange that the spectrum of this central
  multiplier is not the entire unit circle.  We assume this for
  simplicity, and without loss of generality.  Then it has a \(p\)th
  root in the central multiplier algebra of~\(A\).
  Multiplying~\(\varphi\) with the inverse of this \(p\)th root
  gives another isomorphism \(\Hilm_t \cong A\), such that the
  resulting map \(A = \Hilm_{t^p} \congto A\) is the identity.  So
  we get a Fell bundle representation once again.  As above, this
  induces a \Star{}homomorpism \(A\rtimes_\ess S \to A\) that is the
  identity map on~\(A\), and its kernel is a non-zero ideal~\(J\) in
  \(A\rtimes_\ess S\) with \(J\cap A = 0\).
\end{proof}

A version of
Proposition~\ref{pro:intermediate_detection_to_top-free_isg} for
Fell bundles over étale, locally compact groupoids is also true.
The details are left to the reader.

\section{Equivalence of various conditions for crossed products}
\label{sec:various_conditions}

We have proven several implications among properties of a
\(\Cst\)\nb-inclusion \(A\subseteq B\).  We summarised them in the
diagram in Figure~\ref{fig:diagram} (in the introduction) and added a few implications that
are known from previous work.

The left column in Figure~\ref{fig:diagram} is valid for any
\(\Cst\)\nb-inclusion \(A\subseteq B\).  Here we pick a
pseudo-expectation \(E\colon B\to I(A)\) and let~\(\Null\) be the
largest ideal contained in \(\ker E\).  The whole diagram is valid
if~\(B\) is an exotic crossed product for an action on~\(A\) of an
inverse semigroup, or an étale groupoid with locally compact
Hausdorff unit space.  In both cases, the dual groupoid is defined
and there is a natural pseudo-expectation~\(E\) for
\(A\subseteq B\), namely, the canonical generalised expectation
\(B\to \Locmult(A)\) from~\cite{Kwasniewski-Meyer:Essential}
composed with the inclusion \(\Locmult(A) \hookrightarrow I(A)\)
from \cite{Frank:Injective_local_multiplier}*{Theorem~1} (see also
Remark~\ref{rem:ess_defined_vs_pseudo-expectation}).  We use this
pseudo-expectation to define~\(\Null\).  Then~\(B/\Null\) in
Figure~\ref{fig:diagram} is the essential crossed product for the
action, as introduced in~\cite{Kwasniewski-Meyer:Essential}.

Let~\(G\) be an étale groupoid with a locally compact Hausdorff unit
space \(X\).  An action of~\(G\) on the \(\Cst\)\nb-algebra~\(A\)
means a possibly non-saturated Fell bundle~\(\A\) over~\(G\) with
\(A\cong \Cont_0(X,\A|_X)\) (see
Definition~\ref{def:groupoid_action_aperiodic}).  We will call an
action of an étale groupoid purely outer if the corresponding
inverse semigroup action is purely outer.  The claims in
Figure~\ref{fig:diagram} for such actions of~\(G\) follow from the
corresponding claims for inverse semigroup actions.  Fell bundles
over groups are also covered by this.  For Fell bundles over groups
and, more generally, for Fell bundles over \emph{Hausdorff}, étale,
locally compact groupoids, the essential and reduced crossed
products are equal.  More generally, this happens for inverse
semigroup actions if the canonical pseudo-expectation is a genuine
expectation, that is, its values are contained in~\(A\).

Some of the implications in Figure~\ref{fig:diagram} only work if
the pseudo-expectation~\(E\) that is used to define~\(\Null\) has
the following property:

\begin{definition}[\cite{Kwasniewski-Meyer:Essential}*{Definition 3.7}]
  A generalised expectation \(E\colon B\to\tilde{A}\) is
  \emph{symmetric} if \(E(b^* b)=0\) for some \(b\in B\) implies
  \(E(b b^*)=0\).
\end{definition}

A generalised expectation~\(E\) is symmetric if and only if the
largest two-sided ideal~\(\Null\) contained in \(\ker E\) is equal
to \(\setgiven{b\in B}{E(b^* b)=0}\), if and only if the induced
generalised expectation \(B/\Null \to \tilde{A}\) is faithful (see
\cite{Kwasniewski-Meyer:Essential}*{Proposition~3.6 and
  Corollary~3.8}).  The canonical pseudo-expectation on an inverse
semigroup crossed product is always symmetric by
\cite{Kwasniewski-Meyer:Essential}*{Theorem~4.11}.  This remains
true for actions of  étale groupoids
because these are treated by a reduction to inverse semigroup
actions.

We have now explained the meaning of Figure~\ref{fig:diagram}.
Next, we give references for the various implications that are
asserted there.

Proposition~\ref{pro:trivial_isotropy_to_unique_extension} implies
that \(A\subseteq B\) has the extension property if
\(\dual{A}\rtimes S\) is principal and that \(A\subseteq B\) has the
almost extension property if \(\dual{A}\rtimes S\) is topologically
principal.
Proposition~\ref{pro:unique_extension_to_trivial_isotropy} implies
the converse implications provided
\(A\rtimes S \onto B \onto A\rtimes_\red S\), that is, \(B\) sits
between the full and reduced crossed products.  Here we also use
Proposition~\ref{prop:weak_star_density} as in the proof of
Theorem~\ref{the:almost_extension_top_principal}.
Example~\ref{exa:extension_property_not_top_principal} shows that
both converse implications can fail if \(B=A\rtimes_\ess S\).

It is clear that the extension property implies the almost
extension property.  By
Theorem~\ref{the:almost_extension_aperiodic}, any inclusion with the
almost extension property is aperiodic, and the converse holds
if~\(B\) is separable.  It is clear that a principal groupoid is
topologically principal.  Topologically principal groupoids are
topologically free by
\cite{Kwasniewski-Meyer:Essential}*{Lemma~2.23}.  The converse
implication for countable~\(S\)  and \(A\) containing an essential ideal
which is separable or has a Hausdorff spectrum follows from
\cite{Kwasniewski-Meyer:Essential}*{Theorem 6.13 and Corollary 2.6}. 
%, see also the proof of \cite{Kwasniewski-Meyer:Essential}*{Theorem 6.14} and Remark \ref{rem:about_Baire_and_Hausdorff}.
Example~\ref{exa:top_free_vs_principal} shows that the converse
implication may fail for uncountable~\(S\).

Corollary~\ref{cor:top_free_vs_aperiodic} shows that topologically
free inverse semigroup actions are aperiodic.  The inclusion
\(A\subseteq B\) is aperiodic if and only if the action is aperiodic
(see \cite{Kwasniewski-Meyer:Essential}*{Proposition~6.3}, which is
copied here in Proposition~\ref{pro:aperiodic_isg_action}).
Theorem~\ref{thm:aperiodicity_implies_unique_expectation} shows that
an aperiodic inclusion has a unique pseudo-expectation.

If \(A\subseteq B\) is an aperiodic inclusion, then so is
\(A\subseteq C\) for any intermediate \(\Cst\)\nb-algebra
\(A\subseteq C\subseteq B/\Null\) because subbimodules of aperiodic
bimodules remain aperiodic (see
\cite{Kwasniewski-Meyer:Essential}*{Lemma~5.12}).  Therefore, we may
apply Theorem~\ref{thm:aperiodicity_implies_unique_expectation} to
the inclusion \(A\subseteq C\).  If the pseudo-expectation on~\(B\)
is symmetric, then the induced pseudo-expectation
\(B/\Null\to I(A)\) is faithful.  Therefore,
Theorem~\ref{thm:aperiodicity_implies_unique_expectation} implies
that~\(A\) supports~\(C\) for any \(A\subseteq C\subseteq B/\Null\).
It is easy to see that~\(A\) detects ideals in~\(C\) if it
supports~\(C\) (see
\cite{Kwasniewski-Meyer:Essential}*{Lemma~5.27}).  Assume that
\(A\subseteq B\) comes from an inverse semigroup (or étale groupoid)
action and that~\(A\) contains an essential ideal that is separable
or of Type~I.
Proposition~\ref{pro:intermediate_detection_to_top-free_isg} says
that the dual groupoid is topologically free if~\(A\) detects ideals
in all intermediate \(\Cst\)\nb-algebras
\(A\subseteq C\subseteq A\rtimes_\ess S\).

It is shown by Pitts and Zarikian that~\(A\) detects ideals in~\(C\)
for all \(A\subseteq C\subseteq B/\Null\) if and only if all
pseudo-expectations \(B/\Null \to I(A)\) are faithful (see
\cite{Pitts-Zarikian:Unique_pseudoexpectation}*{Theorem~3.5}).
If~\(E\) is symmetric, then the pseudo-expectation
\(B/\Null \to I(A)\) induced by~\(E\) is faithful.  If, in addition,
\(E\) is the only pseudo-expectation \(B\to I(A)\), then it follows
that~\(A\) detects ideals in~\(C\) for all
\(A\subseteq C\subseteq B/\Null\).

Restrict to actions of inverse semigroups once again.  If~\(A\)
detects ideals in all intermediate \(\Cst\)\nb-subalgebras
\(A\subseteq C\subseteq B/\Null\), then the action is purely outer
by Proposition~\ref{pro:intermediate_detection_to_purely_outer}.
By \cite{Kwasniewski-Meyer:Essential}*{Theorem~6.13}, actions that
are aperiodic or topologically free are purely outer, and the
converse holds if~\(A\) contains an essential ideal that is
separable or of Type~I.

This explains all the implications in Figure~\ref{fig:diagram}.

A very important fact is that there are some full cycles of
implications if~\(A\) contains an essential ideal that is separable
or simple or of Type~I.  Therefore, many of the conditions in
Figure~\ref{fig:diagram} become equivalent under suitable
assumptions.  We state two very similar theorems of this type.

\begin{theorem}
  \label{the:cycle_of_equivalences}
  Let~\(A\) be a \(\Cst\)\nb-algebra that contains an essential
  ideal that is separable or of Type~I.  Let \(A\subseteq B\) be the
  inclusion into an exotic crossed product for an action of an
  inverse semigroup~\(S\) or an étale groupoid~\(G\) on~\(A\).
  Assume that the inverse semigroup that acts is countable or that
  the étale groupoid that acts is covered by countably many
  bisections.  Let~\(B_\ess\) be the corresponding essential crossed
  product, which is a quotient of~\(B\).  The following are
  equivalent:
  \begin{itemize}
  \item the dual groupoid of the action is topologically principal;
  \item the dual groupoid of the action is topologically free;
  \item \(A\subseteq B\) has the almost extension property;
  \item \(A\subseteq B\) is aperiodic;
  \item \(A\subseteq B\) has a unique pseudo-expectation;
  \item \(A\) supports all intermediate \(\Cst\)\nb-algebras
    \(A\subseteq C\subseteq B_\ess\);
  \item \(A\) detects ideals in all intermediate
    \(\Cst\)\nb-algebras \(A\subseteq C\subseteq B_\ess\);
  \item \(A\) detects every essential crossed product
    \(A\rtimes_\ess T\) for an inverse subsemigroup \(T\subseteq S\)
    that contains all idempotent elements of~\(S\), or in every
    essential section \(\Cst\)\nb-algebra \(\Cst_\ess(H,\A)\) for an
    open subgroupoid \(H\subseteq G\) that contains the space of
    units of~\(G\).
  \end{itemize}
\end{theorem}

\begin{theorem}
  \label{the:cycle_of_equivalences_2}
  Let~\(A\) be a \(\Cst\)\nb-algebra that contains an essential
  ideal that is simple or of Type~I.  Let \(A\subseteq B\) be the
  inclusion into an exotic crossed product for an action of an
  inverse semigroup~\(S\) or an étale groupoid~\(G\) on~\(A\).
  Let~\(B_\ess\) be the corresponding essential crossed product,
  which is a quotient of~\(B\).  The following are equivalent:
  \begin{itemize}
  \item the action on~\(A\) is purely outer;
  \item \(A\subseteq B\) is aperiodic;
  \item \(A\subseteq B\) has a unique pseudo-expectation;
  \item \(A\) supports all intermediate \(\Cst\)\nb-algebras
    \(A\subseteq C\subseteq B_\ess\);
  \item \(A\) detects ideals in all intermediate
    \(\Cst\)\nb-algebras \(A\subseteq C\subseteq B_\ess\);
  \item \(A\) detects ideals in every essential crossed product
    \(A\rtimes_\ess T\) for an inverse subsemigroup \(T\subseteq S\)
    that contains all idempotent elements of~\(S\), or in every
    essential section \(\Cst\)\nb-algebra \(\Cst_\ess(H,\A)\) for an
    open subgroupoid \(H\subseteq G\) that contains the space of
    units of \(G\).
  \end{itemize}
\end{theorem}

\begin{proof}
  Both theorems follow mostly from the implications in
  Figure~\ref{fig:diagram}.  Under our assumptions, \(E\) is
  symmetric.  Propositions
  \ref{pro:intermediate_detection_to_top-free_isg} and
  \ref{pro:intermediate_detection_to_purely_outer} give topological
  freeness and pure outerness of an inverse semigroup action
  when~\(A\) detects ideals in intermediate \(\Cst\)\nb-algebras of
  the form \(A\rtimes_\ess T\).  The étale groupoid version of
  Proposition~\ref{pro:intermediate_detection_to_top-free_isg} is
  Proposition~\ref{pro:intermediate_detection_to_top-free_groupoid},
  and the étale groupoid version of
  Proposition~\ref{pro:intermediate_detection_to_purely_outer} is
  similar and left to the reader.
\end{proof}

\begin{remark}
  A mistake has crept into the hypotheses of
  \cite{Kwasniewski-Meyer:Aperiodicity}*{Theorem~9.12}.  Namely,
  if~\(A\) contains an essential ideal that is simple, then it is
  unclear whether aperiodicity implies topological freeness (see
  also Remark~\ref{rem:aperiodic_vs_top_free}).  We only know that
  conditions (9.12.1)--(9.12.6) and (9.12.11) in
  \cite{Kwasniewski-Meyer:Aperiodicity}*{Theorem~9.12} are
  equivalent.  Some of these equivalences are shown in
  Theorem~\ref{the:cycle_of_equivalences_2} in greater generality.
\end{remark}

\begin{remark}
  Zarikian proved in
  \cite{Zarikian:Unique_expectations}*{Theorem~3.5} that the
  inclusion \(A\subseteq A\rtimes_\red G\) for an action of a
  discrete group~\(G\) is aperiodic if and only if it has a unique
  pseudo-expectation.  One direction in this implication is proven
  in Theorem~\ref{thm:aperiodicity_implies_unique_expectation}.  We
  do not know whether, conversely, any inclusion with a unique
  pseudo-expectation is aperiodic.  It seems likely that this is
  true for inclusions into crossed products for inverse semigroup
  actions.  Under separability assumptions, this is contained in
  Theorem~\ref{the:cycle_of_equivalences} (see also
  Proposition~\ref{pro:nc_Cartan} below).
\end{remark}

\begin{remark}
  Kennedy and Schafhauser introduced
  in~\cite{Kennedy-Schafhauser:noncomm_boundaries} a cohomological
  invariant for discrete amenable group actions by automorphisms
  whose vanishing implies that aperiodicity (proper outerness) and
  detection of ideals are equivalent.  Using our results, we see
  that this invariant detects whether detection of ideals already
  implies detection of ideals in all intermediate
  \(\Cst\)\nb-subalgebras.
\end{remark}

Many results in the \(\Cst\)\nb-algebra literature have been proven
under one of the assumptions in Figure~\ref{fig:diagram}.  The
implications proven here often allow to strengthen the conclusions
of such results or weaken assumptions.  As an example, we discuss a
classical result of Archbold and
Spielberg~\cite{Archbold-Spielberg:Topologically_free}.  It says
that the inclusion \(A \subseteq A\rtimes_\red G\) for an action of
a discrete group~\(G\) detects ideals if the dual groupoid satisfies
a condition that is between topological freeness and topological
principality, called \emph{AS topologically free}
in~\cite{Kwasniewski-Meyer:Essential} (see also
Remark~\ref{rem:Archbold_Spielberg}).  According to
Figure~\ref{fig:diagram}, we now get the same conclusion whenever
the action is topologically free or aperiodic.  In fact, we get the
stronger statement that~\(A\) supports the reduced crossed product,
which may help to prove that the reduced crossed product is purely
infinite (see
Theorem~\ref{thm:aperiodicity_implies_unique_expectation}).

In our recent paper~\cite{Kwasniewski-Meyer:Essential}, we did not
yet know Theorem~\ref{the:top_free_vs_aperiodic} and therefore
proved results about detection of ideals both for aperiodic and AS
topologically free actions.  Now we see that there is no need for a
separate treatment for (AS) topologically free actions.  So
\cite{Kwasniewski-Meyer:Essential}*{Theorem~6.14} is no longer
needed.

Two implications in Figure~\ref{fig:diagram} only work if
\(E\colon B\to I(A)\) is symmetric.  We can still get similar
statements without this assumption.  These, however, depend
explicitly on the pseudo-expectation~\(E\) and not only on the
ideal~\(\Null\) defined by it:

\begin{lemma}
  Let \(A\subseteq C\subseteq B\) be an intermediate
  \(\Cst\)\nb-algebra.  Let~\(\Null_{E|_C}\) be the largest
  two-sided ideal in~\(C\) contained in \(\ker E\).  If
  \(A\subseteq B\) is aperiodic, then~\(A\) supports
  \(C/\Null_{E|_C}\).  If \(A\subseteq B\) has a unique
  pseudo-expectation, then~\(A\) detects ideals
  in~\(C/\Null_{E|_C}\).
\end{lemma}

\begin{proof}
  If \(A\subseteq B\) is aperiodic or has a unique
  pseudo-expectation, then the same is true for the inclusion
  \(A\subseteq C\).
  Theorem~\ref{thm:aperiodicity_implies_unique_expectation} applied
  to this inclusion shows that~\(A\) supports \(C/\Null_{E|_C}\) if
  \(A\subseteq C\) is aperiodic.  And
  \cite{Pitts-Zarikian:Unique_pseudoexpectation}*{Proposition~3.1}
  shows that~\(A\) detects ideals in~\(C/\Null_{E|_C}\) if
  \(A\subseteq C\) has a unique pseudo-expectation.
\end{proof}

\begin{lemma}
  If~\(E\) is not symmetric, then there is an intermediate
  \(\Cst\)\nb-algebra~\(C\) for which
  \(\Null_{E|_C}\neq\Null_E \cap C\).
\end{lemma}

\begin{proof}
  We are going to construct~\(C\) as the pre-image of an
  intermediate \(\Cst\)\nb-algebra \(A\subseteq C\subseteq B/\Null\)
  such that the pseudo-expectation \(C \to I(A)\) induced by~\(E\)
  is not almost faithful.  The induced pseudo-expectation
  on~\(B/\Null\) is almost faithful, but not faithful.  Then there
  is \(x\in B/\Null\) with \(x\neq 0\), but \(E(x^* x)=0\).  Let
  \(C\) be the \(\Cst\)\nb-subalgebra generated by~\(A\)
  and~\(x^* x\).  It is easy to see that~\(E\) vanishes on the
  two-sided ideal in~\(C\) generated by~\(x^* x\) in~\(C\) (compare
  the proof of
  \cite{Pitts-Zarikian:Unique_pseudoexpectation}*{Theorem~3.5}).
  Therefore, \(\Null_{E|_C}\neq0\).
\end{proof}

\begin{example}
  Let \(B=\Mat_2(\C)\) and let
  \(A \defeq \C\cdot E_{11}\subseteq B\).  It is well known that
  states on hereditary \(\Cst\)\nb-subalgebras extend uniquely (see
  \cite{Pedersen:Cstar_automorphisms}*{Proposition~3.1.6}).
  Since~\(A\) is a hereditary subalgebra in~\(B\), the inclusion
  \(A\subseteq B\) has the extension property.  This implies that it
  is aperiodic and that it has a unique pseudo-expectation.  The
  latter is the obvious expectation \(E\colon B\to A\),
  \((T_{ij})_{1\le i,j\le 2}\mapsto T_{11}\).  This expectation is
  almost faithful because~\(B\) is simple, but not faithful.  So
  \(\Null=0\).  It follows that~\(A\) supports~\(B\).  Now let
  \(C \subseteq B\) be the \(\Cst\)\nb-subalgebra of diagonal
  matrices.  Then~\(E|_C\) is not faithful.  And \(A\subseteq C\)
  does not detect ideals, so it cannot support~\(C\) either.
\end{example}

\section{Applications to  Cartan subalgebras and
  Kumjian's diagonals}
\label{sec:Cartan}

We end this article with two applications of our results to Cartan
subalgebras of some kind.  Since this is the only place where we use
regular inclusions, normalisers, and twists of groupoids, we do not
define these concepts here.
Our regular inclusions are non-degenerate by definition.  First we apply our characterisation of
the extension property to Kumjian's \(\Cst\)\nb-diagonals.  Kumjian
noted that his \(\Cst\)\nb-diagonals have the extension property
(see~\cite{Kumjian:Diagonals}*{Proposition~1.4 and Theorem~3.1}).
The following proposition removes the separability assumption from
this result and shows that the extension property characterises
Kumjian's diagonals.

\begin{proposition}
  \label{pro:diagonal}
  Let \(\Cont_0(X)=A\subseteq B\) be a regular commutative
  \(\Cst\)\nb-subalgebra with a faithful conditional expectation
  \(E\colon B\to A\).  The following are equivalent:
  \begin{enumerate}
  \item \label{pro:diagonal1}%
    \(A\) is a \(\Cst\)\nb-diagonal in~\(B\), that is, normalisers
    \(b\in B\) of~\(A\) with \(b^2=0\) are linearly dense in
    \(\ker E\);
  \item \label{pro:diagonal2}%
    \(A\subseteq B\) has the extension property;
  \item \label{pro:diagonal4}%
    there is a twist~\(\Sigma\) of a principal, Hausdorff, étale
    groupoid~\(H\) with the unit space~\(X\) and an isomorphism
    \(B\cong \Cst_\red(H,\Sigma)\) which is the identity on
    \(A=\Cont_0(X)\).
  \end{enumerate}
  The twisted groupoid \((H,\Sigma)\) in~\ref{pro:diagonal4} is
  unique up to isomorphism.
\end{proposition}

\begin{proof}
  By~\cite{Kumjian:Diagonals}*{Proposition~1.4}, \ref{pro:diagonal1}
  implies~\ref{pro:diagonal2}.  Assume \ref{pro:diagonal2}.  Then
  the inclusion \(A\subseteq B\) is aperiodic by
  Theorem~\ref{the:almost_extension_aperiodic}.  Then
  \cite{Kwasniewski-Meyer:Cartan}*{Corollary~7.6} gives an
  isomorphism \(B\cong \Cst_\red(H,\Sigma)\) for a unique twisted
  groupoid \((H,\Sigma)\), where~\(H\) is a Hausdorff, étale,
  locally compact groupoid with the unit space~\(X\).  Here~\(H\) is
  the dual groupoid of the inclusion.  Since
  \(\Cont_0(X)\subseteq \Cst_\red(H,\Sigma)\) has the extension
  property, \(H\) is also principal by
  Corollary~\ref{cor:extension_prop_Fell_groupoid}.
  Thus~\ref{pro:diagonal2} implies~\ref{pro:diagonal4}.  It remains
  to show that~\ref{pro:diagonal4} implies~\ref{pro:diagonal1} (as
  proven already in \cite{Kumjian:Diagonals}*{Lemma~2.12}).
  Since~\(H\) is principal, \(H \setminus X\) is covered by
  bisections \(U\subseteq H\) with the property that
  \(s(U) \cap r(U) = \emptyset\).  If \(b\in \Cst_\red(H,\Sigma)\)
  is supported in such a bisection, then it is a normaliser for
  \(\Cont_0(X)\) with \(b^2=0\).  Such elements are linearly dense
  in \(\ker E\).
\end{proof}

Next we apply our theory to Exel's noncommutative Cartan
subalgebras, which we have recently characterised
in~\cite{Kwasniewski-Meyer:Cartan}.  A noncommutative Cartan
inclusion is a regular \(\Cst\)\nb-inclusion \(A\subseteq B\) with a
faithful conditional expectation \(E\colon B\to A\) and an extra
property, for which many equivalent forms are given in
\cite{Kwasniewski-Meyer:Cartan}*{Theorem~4.2}.  One of them says
that \(B\cong A\rtimes_\red S\) for a closed and purely outer
action~\(\Hilm\) of an inverse semigroup~\(S\) on~\(A\), with an
isomorphism that restricts to the canonical embedding on~\(A\).  It
is already noted in \cite{Kwasniewski-Meyer:Cartan}*{Theorem~6.3}
that aperiodicity implies the equivalent conditions in
\cite{Kwasniewski-Meyer:Cartan}*{Theorem~4.2}.  Here we observe that
the unique pseudo-expectation property also does that.  This
characterises Renault's (commutative) Cartan inclusions in a new way
(see~\cite{Renault:Cartan.Subalgebras}).  It also characterises
Exel's Cartan subalgebras \(A\subseteq B\), provided~\(A\) has an
essential ideal that is simple or of Type~I.

\begin{proposition}
  \label{pro:nc_Cartan}
  Let \(A\subseteq B\) be a regular \(\Cst\)\nb-inclusion with a
  faithful conditional expectation \(E\colon B\to A\).  Consider the
  following conditions:
  \begin{enumerate}
  \item \label{enu:noncommutative_Cartans0}%
    \(A\subseteq B\) has the almost extension property;
  \item \label{enu:noncommutative_Cartans1}%
    \(A\subseteq B\) is aperiodic;
  \item \label{enu:noncommutative_Cartans2}%
    \(A\subseteq B\) has a unique pseudo-expectation;
  \item \label{enu:noncommutative_Cartans3}%
    \(A\subseteq B\) is a noncommutative Cartan subalgebra in the
    sense of Exel~\cite{Exel:noncomm.cartan};
  \end{enumerate}
  Then \ref{enu:noncommutative_Cartans0}\(\Rightarrow\)%
  \ref{enu:noncommutative_Cartans1}\(\Rightarrow\)%
  \ref{enu:noncommutative_Cartans2}\(\Rightarrow\)%
  \ref{enu:noncommutative_Cartans3}.  If~\(A\) contains an essential
  ideal that is simple or of Type~I, then
  \ref{enu:noncommutative_Cartans1}--\ref{enu:noncommutative_Cartans3}
  are equivalent.  If, in addition, \(B\) is separable, then all the
  conditions
  \ref{enu:noncommutative_Cartans0}--\ref{enu:noncommutative_Cartans3}
  are equivalent.
\end{proposition}

\begin{proof}
  Recall that if~\ref{enu:noncommutative_Cartans3} holds, then
  \(B\cong A\rtimes_\red S=A\rtimes_\ess S\) for a closed purely
  outer action of an inverse semigroup~\(S\).  Hence all the
  implications except
  \ref{enu:noncommutative_Cartans2}\(\Rightarrow\)%
  \ref{enu:noncommutative_Cartans3} are included in
  Figure~\ref{fig:diagram}.  Implication
  \ref{enu:noncommutative_Cartans2}\(\Rightarrow\)%
  \ref{enu:noncommutative_Cartans3} follows from
  \cite{Kwasniewski-Meyer:Cartan}*{Theorem~4.2} and the following
  lemma, which checks condition~(1) of
  \cite{Kwasniewski-Meyer:Cartan}*{Theorem~4.2}.
\end{proof}

\begin{lemma}
  \label{lem:unique_pseudo_ideals}
  Let \(A\subseteq B\) be a \(\Cst\)\nb-inclusion with a unique
  pseudo-expectation.  For every ideal~\(I\) in~\(A\), there is at
  most one conditional expectation for the inclusion
  \(I\subseteq IBI\).
\end{lemma}

\begin{proof}
  Let~\(I\) be an ideal in~\(A\) and let \(E\colon IBI\to I\) be a
  conditional expectation.  We are going to prove that~\(E\) extends
  to a pseudo-expectation \(B\to I(A)\).  Since the latter is
  unique, it follows that~\(E\) is also unique.

  Let~\(I^\bot\) be the annihilator of~\(I\) in~\(A\).  We
  extend~\(E\) to a conditional expectation for the inclusion
  \(I+I^\bot \subseteq IBI+I^\bot\) by putting
  \(E|_{I^\bot}\defeq \Id_{I^\bot}\).  This extends to a strictly
  continuous conditional expectation
  \(\overline{E}\colon \Mult (IBI+I^\bot)\to \Mult(I+I^\bot)\) by
  \cite{Kwasniewski-Meyer:Cartan}*{Lemma 4.6}.  Since \(IBI+I^\bot\)
  is an essential ideal in \(IBI+A\) we may treat the latter as a
  subalgebra of \(\Mult (IBI+I^\bot)\).  In this way, we get a
  generalised conditional expectation
  \(\overline{E}\colon IBI+I^\bot\to \Mult(I+I^\bot)\).  Since
  \(I+I^\bot\) is an essential ideal in~\(A\) we have
  \(\Mult(I+I^\bot)\subseteq \Locmult(A) \subseteq I(A)\).  Since
  \(I(A)\) is injective, the map
  \(\overline{E}\colon B\supseteq IBI+A\to \Mult(I+I^\bot)\subseteq
  I(A)\) extends to a pseudo-expectation \(B\to I(A)\).
\end{proof}

\end{document}